\documentclass{amsart}

\usepackage[utf8]{inputenc}
\usepackage{amsmath}
\usepackage{amssymb}
\usepackage{amsthm}
\usepackage[style=alphabetic,maxbibnames=8,safeinputenc,backend=biber]{biblatex}

\usepackage[hyperfootnotes=false]{hyperref}
\usepackage{enumerate}
\usepackage{graphicx}
\usepackage{tikz}
\usepackage{tikz-cd}
\usetikzlibrary{decorations.pathreplacing,calc,tikzmark}
\usepackage[labelfont=bf]{caption}
\usepackage{accents}
\usepackage{microtype}

\usepackage{todonotes}
\usepackage{fancyhdr}

\newcommand{\real}{\mathbb{R}}
\newcommand{\ints}{\mathbb{Z}}

\newcommand{\nats}{\mathbb{N}}

\newcommand{\proj}{\mathbf{P}}
\newcommand{\Gr}{\mathrm{Gr}}

\newcommand{\Isom}{\mathrm{Isom}}
\DeclareMathOperator{\PGL}{PGL}
\DeclareMathOperator{\GL}{GL}
\DeclareMathOperator{\SL}{SL}

\DeclareMathOperator{\SO}{SO}

\newcommand{\eps}{\epsilon}
\newcommand{\del}{\partial}
\newcommand{\into}{\hookrightarrow}

\DeclareMathOperator{\id}{id}

\DeclareMathOperator{\Cay}{Cay}
\DeclareMathOperator{\Stab}{Stab}

\DeclareMathOperator{\Ad}{Ad}

\newcommand{\ubar}[1]{\underaccent{\bar}{#1}}

\theoremstyle{plain}
\newtheorem{thm}{Theorem}[section]
\newtheorem{lem}[thm]{Lemma}
\newtheorem{prop}[thm]{Proposition}
\newtheorem{cor}[thm]{Corollary}
\newtheorem{defn}[thm]{Definition}

\newtheorem{introthm}{Theorem}

\newtheorem{introprop}[introthm]{Proposition}

\theoremstyle{definition}
\newtheorem{eg}[thm]{Example}
\newtheorem{rmk}[thm]{Remark}
\newtheorem*{claim}{Claim}

\title{Relatively dominated representations from eigenvalue gaps and limit maps}
\author{Feng Zhu}

\begin{document}

\begin{abstract}
Relatively dominated representations give a common generalization of geometrically finiteness in rank one on the one hand, and the Anosov condition which serves as a higher-rank analogue of convex cocompactness on the other. This note proves three results about these representations.

Firstly, we remove the technical assumption of quadratic gaps involved in the original definition. Secondly, we give a characterization using eigenvalue gaps, providing a relative analogue of a result of Kassel--Potrie for Anosov representations. Thirdly, we formulate characterizations in terms of singular value or eigenvalue gaps combined with limit maps, in the spirit of Gu\'eritaud--Guichard--Kassel--Wienhard for Anosov representations, and use them to show that inclusion representations of certain groups playing weak ping-pong are relatively dominated.
\end{abstract}

\maketitle

\section{Introduction}
Relatively dominated representations were introduced in \cite{reldomreps} and provide a common generalization of geometric finiteness in rank-one semisimple Lie groups and the Anosov condition in more general semisimple Lie groups. They are related to earlier common generalizations studied by Kapovich and Leeb in \cite{KL}. These representations furnish a class of discrete relatively hyperbolic subgroups of semisimple Lie groups which are quasi-isometrically embedded modulo controlled distortion along their peripheral subgroups.

The definition of these representations is given in terms of singular value gaps, which may be interpreted in terms of the geometry of the associated symmetric spaces as distances from singular flats of specified type. The corresponding characterization of Anosov representations was given first by Kapovich--Leeb--Porti in \cite{KLP} under the name of URU subgroups, and subsequently reformulated, in language more closely resembling that used here, by Bochi--Potrie--Sambarino in \cite{BPS}. 

The key defining condition asserts that the singular value gap $\frac{\sigma_1}{\sigma_2}(\rho(\gamma))$ grows uniformly exponentially in a notion of word-length $|\gamma|_c$ that has been modified to take into account the distortion along the peripheral subgroups.

The definition also involves additional technical conditions to control the images of the peripheral subgroups. In the first part of this note, we remove one of those technical conditions, by showing that its relevant consequences also follow from other parts of the definition. We refer the reader to \S\ref{sec:D+-EC}, and specifically Proposition \ref{prop:D+-_EC}, for the full statement; here we present it slightly summarised as follows:
\begin{introprop} \label{introprop:3.5}
Let $\Gamma$ be a finitely-generated group, $\mathcal{P}$ a collection of finitely-generated subgroups such that we can build a cusped space $X(\Gamma,\mathcal{P})$ and obtain a cusped word-length $d_c$, and fix $C_0 > 0$. Suppose we have a representation $\rho\colon \Gamma \to \SL(d,\real)$ such that for all $\gamma \in \Gamma$,
$$C_0^{-1} \log \frac{\sigma_1}{\sigma_2}(\rho(\gamma)) - C_0 \leq |\gamma|_c \leq C_0 \log \frac{\sigma_1}{\sigma_2}(\rho(\gamma)) + C_0 .$$

Then, given constants $\ubar\upsilon, \bar\upsilon > 0$, there exists constants $C, \mu > 0$ such that for any bi-infinite sequence of elements $(\gamma_n)_{n \in \ints} \subset \Gamma$ satisfying 
\begin{enumerate}[(i)]
    \item $\gamma_0 = \id$, and
    \item $\ubar\upsilon^{-1}|n| -\ubar\upsilon \leq |\gamma_n|_c \leq \bar\upsilon|n| + \bar\upsilon$ for all $n$,
\end{enumerate}
and any $k \in \ints$, 
$$d \left( U_1(\rho(\gamma_{k-1} \cdots \gamma_{k-n})), U_1(\rho(\gamma_{k-1} \cdots \gamma_{k-n-1})) \right) < Ce^{-\mu n} $$
for all $n >0$.
\end{introprop}
Here 
$U_1(B)$ denotes the image of the 1-dimensional subspace of $\real^d$ most expanded by $B$. This proposition allows us to obtain uniform convergence towards limit points; the exponential convergence seen here is reminiscent of phenomena from hyperbolic dynamics, and is straightforward to obtain in the non-relative case.

In the proof of Proposition \ref{prop:D+-_EC} we will find it useful to adopt elements of the point of view of Kapovich--Leeb--Porti, which emphasizes the geometry of the symmetric space and the related geometry of its boundary and associated flag spaces. 

More recently, Kassel--Potrie \cite{KasselPotrie} have given a characterization of Anosov representations in terms of eigenvalue gaps $\frac{\lambda_i}{\lambda_{i+1}}$, which may be interpreted as asymptotic versions of singular value gaps $\frac{\sigma_i}{\sigma_{i+1}}$, i.e.\ distance to the Weyl chamber walls at infinity. In the second part of this note, we give an analogous characterization of relatively dominated representations:
\begin{introthm}[Corollary \ref{cor:reldom_eiggap}] \label{thm:intro_reldom_eiggap}
Let $\Gamma$ be hyperbolic relative to $\mathcal{P}$.
A semisimple representation $\rho\colon \Gamma \to \SL(d,\real)$ is $P_1$-dominated relative to $\mathcal{P}$ if and only if the following four conditions hold:
\begin{itemize}
\item (${D}^\lambda_-$) there exist constants $\ubar{C}, \ubar\mu > 0$ such that
$$\frac{\lambda_1}{\lambda_2}(\rho(\gamma)) \geq \ubar{C} e^{\ubar\mu |\gamma|_{c,\infty}}$$
for all $\gamma \in \Gamma$, 
\item (${D}^\lambda_+$) there exist constants $\bar{C}, \bar\mu > 0$ such that
$$\frac{\lambda_1}{\lambda_d}(\rho(\gamma)) \leq \bar{C} e^{\bar\mu |\gamma|_{c,\infty}}$$
for all $\gamma \in \Gamma$,
\item (unique limits) for each $P \in \mathcal{P}$, there exists $\xi_\rho(P) \in \proj(\real^d)$ and  $\xi^*_\rho(P) \in \Gr_{d-1}(\real^d)$ such that for every sequence $(\eta_n) \subset P$ with $\eta_n \to \infty$, we have $\lim_{n\to\infty} U_1(\rho(\eta_n)) = \xi_\rho(P)$ and $\lim_{n\to\infty} U_{d-1}(\rho(\eta_n)) = \xi^*_\rho(P)$.

\item (uniform transversality) 
for every $P, P' \in \mathcal{P}$ and $\gamma \in \Gamma$, $\xi_\rho(P) \neq \xi_\rho(\gamma P'\gamma^{-1})$. 
Moreover,
for every $\ubar\upsilon,\bar\upsilon>0$, 
there exists $\delta_0 > 0$ 
such that for all $P, P' \in \mathcal{P}$ and $g, h \in \Gamma$
such that there exists a bi-infinite $(\ubar\upsilon,\bar\upsilon)$-metric quasigeodesic path $\eta gh \eta'$ where 
$\eta'$ is in $P'$ and $\eta$ is in $P$, 
we have 
\[ \sin \angle (g^{-1} \xi_\rho(P), h\, \xi^*_\rho(P')) > \delta_0 .\]
{\footnotesize (See Definition \ref{defn:metric_proj_qgeod} for the precise definition of a metric quasigeodesic path.)}
\end{itemize} \end{introthm}

Here $|\gamma|_{c,\infty}$ is a stable version of the modified word-length $|\gamma|_c$; we refer the reader to \S\ref{sub:stablen_grop} for the precise definitions.
Note there is an additional semisimplicity assumption here; there are additional subtleties that arise in the relative case which make a more unqualified statement tricky. We recall that a representation into $\SL(d,\real)$ is called semisimple if the Zariski closure of its image is a reductive group. Equivalently, semisimple representations may be written as direct sums of irreducible representations.

The proof of Theorem \ref{thm:intro_reldom_eiggap} leverages a recent result of Tsouvalas \cite[ Th.\,5.3]{Kostas_AnosovWEG} stating that groups admitting non-trivial Floyd boundaries have property U: this property, roughly speaking, allows us to control stable translation lengths in terms of word-length. Relatively hyperbolic groups admit non-trivial Floyd boundaries (\cite{GerFloyd}, see also Remark \ref{rmk:floyd_relhyp}), and here we establish a modified version of property U adapted to the relatively hyperbolic case. 

Finally, we present characterizations of relatively dominated representations which replace most of the additional conditions on the peripheral images with conditions about the existence of suitable limit maps. These are relative analogues of results due to Gu\'eritaud--Guichard--Kassel--Wienhard \cite{GGKW}. 

\begin{introthm} [Theorem \ref{thm:gaps+maps} + Corollaries \ref{thm:eiggaps+maps_a} \& \ref{thm:eiggaps+maps_b}] \label{thm:intro_gaps+maps}
Given $(\Gamma,\mathcal{P})$ a relatively hyperbolic group, a representation $\rho\colon \Gamma \to \SL(d,\real)$ is $P_1$-dominated relative to $\mathcal{P}$ if and only if
\begin{itemize}
\item there exist continuous, $\rho(\Gamma)$-equivariant, transverse, dynamics-preserving limit maps 
$\xi_\rho\colon \del(\Gamma,\mathcal{P}) \to \proj(\real^d)$ and $\xi_\rho^*\colon \del(\Gamma,\mathcal{P}) \to \proj(\real^{d*})$,
\end{itemize}
and one of the following sets of conditions holds:
\begin{itemize}
\item \emph{either} there exist constants $\ubar{C}, \ubar\mu > 0$ and $\bar{C},\bar\mu > 0$ such that 
\begin{itemize}
    \item[(D$-$)] 
    $\frac{\sigma_1}{\sigma_{2}}(\rho(\gamma)) \geq \ubar{C} e^{\ubar\mu|\gamma|_c}$ for all $\gamma \in \Gamma$, and
    \item[(D+)] 
    $\frac{\sigma_1}{\sigma_d}(\rho(\gamma)) \leq \bar{C} e^{\bar\mu|\eta|_c}$ for all $\gamma \in \Gamma$; 
\end{itemize} 
\item \emph{or} there exist constants $\ubar{C}, \ubar\mu > 0$ and $\bar{C},\bar\mu > 0$ such that 
\begin{itemize}
\item[(D${}^\lambda_-$)]
$\frac{\lambda_1}{\lambda_2}(\rho(\gamma)) \geq \ubar{C} e^{\ubar\mu |\gamma|_{c,\infty}}$
for all $\gamma \in \Gamma$, and
\item[(D${}^\lambda_+$)]
$\frac{\lambda_1}{\lambda_d}(\rho(\gamma)) \leq \bar{C} e^{\bar\mu |\gamma|_{c,\infty}}$
for all $\gamma \in \Gamma$.
\end{itemize}
\emph{and} at least one of the following holds: 
\begin{itemize} 
\item $\rho$ is semisimple, or
\item $\xi_\rho$ is a homeomorphism onto 
$$
\Lambda_1(\rho) := \bigcap_{n \geq \ell_0} \overline{\{U_1(\rho(\gamma)) : |\gamma|_c \geq n\}}
$$
{\footnotesize where $\ell_0 \in \ints_{>0}$ is large enough so that $U_1(\rho(\gamma))$ is well-defined whenever $|\gamma|_c \geq \ell_0$.}
\end{itemize} 
\end{itemize}
\end{introthm}

Here, $\xi$ and $\xi^*$ are said to be {\bf transverse} if $\xi(x) \oplus \xi^*(y) = \real^d$ for all $x \neq y$, 
and they are said to be {\bf dynamics-preserving} if 
\begin{enumerate}[(i)]
\item $\xi(\gamma^+) = (\rho(\gamma))^+$ and $\xi^*(\gamma^+)^\perp = (\rho^*(\gamma))^+$.
for all nonperipheral $\gamma \in \Gamma$, where $\gamma^+ := \lim_{n\to\infty} \gamma^n \in \del(\Gamma,\mathcal{P})$ and $\rho(\gamma)^+$ is the attracting eigenline for $\rho(\gamma)$,
and
\item If $\del P \in \del(\Gamma,P)$ is the unique point associated to $P \in \mathcal{P}$, then $\xi^{(*)}(\del P)$ is the parabolic fixed point in $\proj(\real^{d(*)})$ associated to $\rho^{(*)}(P)$ (where $\rho^*\colon \Gamma \to \SL(d,\real)$ is the dual representation defined by $\rho^*(\gamma) := (\rho(\gamma^{-1}))^T$). In particular, these fixed points exist and are well-defined.
\end{enumerate}

As an application of this, we show that certain free groups which contain unipotent generators and which play weak ping-pong in projective space are relatively $P_1$-dominated (Example \ref{eg:proj_schottky} below).

\subsection*{Organization} 
\S\ref{sec:prelim} collects the various preliminaries needed. \S\ref{sec:D+-EC} gives the definition of a relatively dominated representation, with the simplification allowed by Proposition \ref{introprop:3.5} / \ref{prop:D+-_EC}. \S\ref{sec:reldom_eiggap} contains the proof of the eigenvalue gaps + peripheral conditions characterization described in Theorem \ref{thm:intro_reldom_eiggap}, and \S\ref{sec:gaps+maps} contains the proofs of the gaps + limit maps characterizations described in Theorem \ref{thm:intro_gaps+maps}, as well as their application to weak ping-pong groups. Note the only dependence of \S\ref{sec:reldom_eiggap} and \S\ref{sec:gaps+maps} on \S\ref{sec:D+-EC} is in the definition of relatively dominated representations.

\subsection*{Acknowledgements}
The author thanks Max Riestenberg for helpful conversations about the Kapovich--Leeb--Porti approach to Anosov representations, Kostas Tsouvalas for stimulating comments, Fran\c{c}ois Gu\'eritaud and Jean-Philippe Burelle for helpful discussions related to ping-pong and positive representations, Andrew Zimmer for pointing out a missing hypothesis and a gap in a reference, and Fanny Kassel for comments on an earlier version.
The author acknowledges support from ISF grants 871/17 and 737/20.

This research was conducted during the COVID-19 pandemic. The author extends his heartfelt gratitude to all those --- friends, family, mentors, funding agencies --- who have given him safe harbor in these tumultuous times.

\section{Preliminaries} \label{sec:prelim}

\subsection{Relatively hyperbolic groups and cusped spaces}

Relative hyperbolicity is a group-theoretic notion
of non-positive curvature inspired by the geometry of cusped hyperbolic manifolds and free products.  

Consider a finite-volume cusped hyperbolic manifold with an open neighborhood of each cusp removed: call the resulting truncated manifold $M$. The universal cover $\tilde{M}$ of such a $M$ is hyperbolic space with a countable set of horoballs removed. The universal cover $\tilde{M}$ is not Gromov-hyperbolic; distances along horospheres that bound removed horoballs are distorted. If we glue the removed horoballs back in to the universal cover, however, the resulting space will again be hyperbolic space.

Gromov generalized this in \cite[\S8.6]{Gromov} by defining a group $\Gamma$ as hyperbolic relative to a conjugation-invariant collection of subgroups $\mathcal{P}$ if $(\Gamma,\mathcal{P})$ admits a {\bf cusp-uniform action} on a (Gromov-)hyperbolic metric space $X$, meaning there exists some system $(\mathcal{H}_P)_{P\in\mathcal{P}}$ of disjoint horoballs of $X$, each preserved by a subgroup $P \in \mathcal{P}$, such that the $\Gamma$ acts on $X$ discretely and isometrically, and the $\Gamma$ action on $X \smallsetminus \bigcup_P \mathcal{H}_P$ is cocompact. 

The hyperbolic space $X$ is sometimes called a Gromov model for $(\Gamma,\mathcal{P})$. There is in general no canonical Gromov model for a given relatively hyperbolic group, but there are systematic constructions one can give, one of which we describe here. The description below, as well as the material in the next section \S\ref{sub:hat_unhat}, is taken from \cite[\S2]{reldomreps} and is based on prior literature, in particular \cite{GrovesManning}; it is included here for completeness.

\begin{defn}[{\cite[Def.\,3.1]{GrovesManning}}] \label{defn:combhoroball}
Given a subgraph $\Lambda$ of the Cayley graph $\Cay(\Gamma,S)$, the {\bf combinatorial horoball} based on $\Lambda$, denoted $\mathcal{H} = \mathcal{H}(\Lambda)$, is the 1-complex\footnote{Groves-Manning combinatorial horoballs are actually defined as 2-complexes; the definition here is really of a 1-skeleton of a Groves-Manning horoball. For metric purposes only the 1-skeleton matters.} formed as follows:
\begin{itemize}
\item the vertex set $\mathcal{H}^{(0)}$ is given by $\Lambda^{(0)} \times \ints_{\geq 0}$
\item the edge set $\mathcal{H}^{(1)}$ consists of the following two types of edges: \begin{enumerate}[(1)]
\item if $k \geq 0$ and $v$ and $w \in \Lambda^{(0)}$ are such that $0 < d_\Lambda(v, w) \leq 2^k$, then there is a (``horizontal'') edge connecting
$(v, k)$ to $(w, k)$; 
\item if $k \geq 0$ and $v \in \Lambda^{(0)}$, there is a (``vertical'') edge joining $(v, k)$ to $(v, k + 1)$.
\end{enumerate} \end{itemize}
$\mathcal{H}$ is metrized by assigning length 1 to all edges.
\end{defn}

Next let $\mathcal{P}$ be a finite collection of finitely-generated subgroups of $\Gamma$, and suppose $S$ is a {\bf compatible generating set}, i.e.\ for each $P \in \mathcal{P}$, $S \cap P$ generates $P$. 

\begin{defn}[{\cite[Def.\,3.12]{GrovesManning}}] \label{defn:cuspedspace}
Given $\Gamma, \mathcal{P}, S$ as above, the {\bf cusped space} $X(\Gamma, \mathcal{P}, S)$ is the simplicial metric graph 
\[ \Cay(\Gamma,S) \cup \bigcup \mathcal{H}(\gamma P) \]
where the union is taken over all left cosets of elements of $\mathcal{P}$, i.e.\ over $P \in \mathcal{P}$ and (for each $P$) $\gamma P$ in a collection of representatives for left cosets of $P$. 

Here the induced subgraph of $\mathcal{H}(\gamma P)$ on the $\gamma P \times \{0\}$ vertices is identified with (the induced subgraph of) $\gamma P \subset \Cay(\Gamma,S)$ in the natural way.
\end{defn}

\begin{defn} \label{defn:relhyp}
$\Gamma$ is hyperbolic relative to $\mathcal{P}$ if and only if the cusped space $X(\Gamma,\mathcal{P},S)$ is hyperbolic (for any compatible generating set $S$; the hyperbolicity constant may depend on $S$.) 

We will also call $(\Gamma, \mathcal{P})$ a {\bf relatively hyperbolic structure}.
\end{defn}

We remark that for a fixed relatively hyperbolic structure $(\Gamma, \mathcal{P})$, any two cusped spaces, corresponding to different compatible generating sets $S$, are quasi-isometric \cite[Cor.\,6.7]{Groff}:
in particular, the notion above is well-defined independent of the choice of generating set $S$.
There is a natural action of $\Gamma$ on the cusped space $X = X(\Gamma,\mathcal{P},S)$;
with respect to this action, the quasi-isometry between two cusped spaces $X(\Gamma,\mathcal{P},S_i)$ ($i=1,2$) is $\Gamma$-equivariant. 


In particular, this gives us a notion of a boundary associated to the data of a relatively hyperbolic group $\Gamma$ and its peripheral subgroups $\mathcal{P}$:
\begin{defn} \label{defn:bowditch_bdy}
For  $\Gamma$ hyperbolic relative to $\mathcal{P}$, the {\bf Bowditch boundary} $\del (\Gamma, \mathcal{P})$ is defined as the Gromov boundary $\del_\infty X$ of any cusped space $X = X(\Gamma,\mathcal{P},S)$.
\end{defn}
This boundary is well-defined up to homeomorphism, independent of the choice of compatible generating set $S$ \cite[\S9]{Bowditch}.

Below, with a fixed choice of $\Gamma$, $\mathcal{P}$ and $S$ as above, for $\gamma, \gamma' \in \Gamma$, $d(\gamma, \gamma')$ will denote the distance between $\gamma$ and $\gamma'$ in the Cayley graph with the word metric, and $|\gamma| := d(\id, \gamma)$ denotes word length in this metric. Similarly, $d_c(\gamma, \gamma')$ denotes distance in the corresponding cusped space and $|\gamma|_c := d_c(\id,\gamma)$ denotes cusped word-length.

\subsection{Geodesics in the cusped space} \label{sub:hat_unhat}

Let $\Gamma$ be a finitely-generated group, $\mathcal{P}$ be a malnormal finite collection of finitely-generated subgroups, and let $S = S^{-1}$ be a compatible finite generating set as above. Let $X = X(\Gamma, \mathcal{P}, S)$ be the cusped space, and $\Cay(\Gamma) = \Cay(\Gamma,S)$ the Cayley graph. Here we collect some technical results about geodesics in these spaces that will be useful below. 

\begin{lem}[{\cite[Lem.\,3.10]{GrovesManning}}] \label{lem:gm310}
Let $\mathcal{H}(\Gamma)$ be a combinatorial horoball. Suppose that $x,y \in \mathcal{H}(\Gamma)$ are distinct vertices. Then there is a geodesic $\gamma(x,y) = \gamma(y,x)$ between $x$ and $y$ which consists of at most two vertical segments and a single horizontal segment of length at most 3. 
\end{lem}

We will call any such geodesic a {\bf preferred geodesic}.



Given a path $\gamma\colon I \to \Cay(\Gamma)$ in the Cayley graph 
such that $\gamma(I \cap \ints) \subset \Gamma$, we can consider $\gamma$ as a {\bf relative path}
$(\gamma, H)$, where $H$ is a subset of $I$ consisting of a disjoint union of finitely many subintervals $H_1, \dots, H_n$  occurring in this order along $I$, such that each $\eta_i := \gamma|_{H_i}$ is a maximal subpath lying in a closed combinatorial horoball $B_i$, 
and $\gamma|_{I \smallsetminus H}$ contains no edges of $\Cay(\Gamma)$ labelled by a peripheral generator.

Similarly, a path $\hat\gamma\colon \hat{I} \to X$ in the cusped space with endpoints in $\Cay(\Gamma) \subset X$ may be considered as a relative path $(\hat\gamma, \hat{H})$, where $\hat{H} = \coprod_{i=1}^n \hat{H}_i$, $\hat{H}_1, \dots, \hat{H}_n$ occur in this order along $\hat{I}$, each $\hat\eta_i := \hat\gamma|_{\hat{H}_i}$ is a maximal subpath in a closed combinatorial horoball $B_i$, and $\hat\gamma|_{\hat{I} \smallsetminus \hat{H}}$ lies inside the Cayley graph. Below, we will consider only geodesics and quasigeodesic paths $\hat\gamma\colon \hat{I} \to X$ where all of the $\hat\eta_i$ are preferred geodesics (in the sense of Lemma \ref{lem:gm310}.)

We will refer to the $\eta_i$ and $\hat\eta_i$ as {\bf peripheral excursions}. We remark that the $\eta_i$, or any other subpath of $\gamma$ in the Cayley graph, may be considered as a word and hence a group element in $\Gamma$; this will be used without further comment below.

Given a path $\hat\gamma\colon \hat{I} \to X$ whose peripheral excursions are all preferred geodesics, we may replace each excursion $\hat\eta_i = \hat\gamma|_{\hat{H}_i}$ into a combinatorial horoball with a geodesic path (or, more precisely, a path with geodesic image) $\eta_i = \pi \circ \hat\eta_i$ in the Cayley (sub)graph of the corresponding peripheral subgroup connecting the same endpoints, by omitting the vertical segments of the preferred geodesic $\hat\eta_i$ and replacing the horizontal segment with the corresponding segment at level 0, i.e.\ in the Cayley graph.\footnote{As a parametrized path this has constant image on the subintervals of $\hat{H}_i$ corresponding to the vertical segments, and travels along the projected horizontal segment at constant speed.} We call this the ``project'' operation, since it involves ``projecting'' paths inside combinatorial horoballs onto the boundaries of those horoballs. This produces a path $\gamma = \pi\circ\hat\gamma\colon \hat{I} \to \Cay(\Gamma)$. 


Given any path $\alpha$ in the Cayley graph with endpoints $g, h \in \Gamma$, we write $\ell(\alpha)$ to denote $d(g,h)$, i.e.\ distance measured according to the word metric in $\Cay(\Gamma)$.

We have the following biLipschitz equivalence between cusped distances and suitably-modified distances in the Cayley graph:
\begin{prop}[{\cite[Prop.\ 2.12]{reldomreps}}] \label{prop:unhat_distance}
Given a geodesic $\hat\gamma\colon \hat{J} \to X$ with endpoints in $\Cay(\Gamma) \subset X$ and whose peripheral excursions are all preferred geodesics, let $\gamma = \pi \circ \hat\gamma\colon \hat{J} \to \Cay(\Gamma)$ be its projected image. 

Given any subinterval $[a,b] \subset \hat{J}$, consider the subpath $\gamma|_{[a,b]}$ as a relative path $(\gamma|_{[a,b]}, H)$ where $H = (H_1, \dots, H_n)$, and write $\eta_i := \gamma|_{H_i}$; then we have 
\[ \frac13 \leq \frac{d_c(\gamma(a), \gamma(b))}{\ell(\gamma|_{[a,b]}) - \sum_{i=1}^n \ell(\eta_i) + \sum_{i=1}^n \hat\ell(\eta_i)} \leq \frac{2}{\log 2} + 1 < 4 \]
where $\hat\ell(\eta_i) := \max\{\log(\ell(\eta_i)), 1\}$.
\end{prop}

Below we will occasionally find it useful to consider paths in $\Cay(\Gamma)$ that ``behave metrically like quasi-geodesics in the relative Cayley graph'', in the following sense: 
\begin{defn} \label{defn:projgeod_depth}
Given any path $\gamma\colon I \to \Cay(\Gamma)$ such that $I$ has integer endpoints and $\gamma(I \cap \ints) \subset \Gamma$,
define the {\bf depth} $\delta(n) = \delta_\gamma(n)$ of a point $\gamma(n)$ (for any $n \in I \cap \ints$) as \begin{enumerate}[(a)]
\item the smallest integer $d \geq 0$ such that at least one of $\gamma(n-d)$, $\gamma(n+d)$ is well-defined (i.e.\ $\{n-d, n+d\} \cap I \neq \varnothing$) and not in the same peripheral coset as $\gamma(n)$, {\bf or}
\item if no such integer exists, $\min\{\sup I - n, n - \inf I\}$.
\end{enumerate}
\end{defn}

\begin{defn} \label{defn:metric_proj_qgeod}
Given constants $\ubar\upsilon, \bar\upsilon > 0$, an {\bf $(\ubar\upsilon,\bar\upsilon)$-metric quasigeodesic path} is a path $\gamma\colon I \to \Cay(\Gamma)$ 
with $\gamma(I \cap \ints) \subset \Gamma$
such that
for all integers $m, n \in I$, 
\begin{enumerate}[(i)]
\item $ |\gamma(n)^{-1}\gamma(m)|_c \geq \ubar\upsilon^{-1} |m-n| - \ubar\upsilon$, 
\item $ |\gamma(n)^{-1}\gamma(m)|_c \leq \bar\upsilon(|m-n| + \min\{\delta(m), \delta(n)\} ) + \bar\upsilon$, and
\item if $\gamma(n)^{-1} \gamma(n+1) \in P$ for some $P \in \mathcal{P}$, we have $\gamma(n)^{-1} \gamma(n+1) = p_{n,1} \cdots p_{n,\ell(n)}$ where each $p_{n,i}$ is a peripheral generator of $P$, and $$2^{\delta(n)-1} \leq \ell(n) := |\gamma(n)^{-1}\gamma(n+1)| \leq 2^{\delta(n)+1}.$$ 
\end{enumerate}
\end{defn}

The terminology comes from the following fact: given a geodesic segment $\hat\gamma$ in the cusped space with endpoints in $\Cay(\Gamma)$, we can project the entire segment to the Cayley graph and reparametrize the projected image to be a metric quasigeodesic path --- the idea being that in such a reparametrization, the increments correspond, approximately, to linear increments in cusped distance: see the discussion in \cite[\S2.3]{reldomreps}, and in particular Prop.\ 2.16 there for more details.

\subsection{Floyd boundaries} \label{sub:Floyd_bdy}

Let $\Gamma$ be a finitely-generated group, and $S$ a finite generating set giving a word metric $|\cdot|$. 

A {\bf Floyd boundary} $\del_f \Gamma$ for $\Gamma$ is a boundary for $\Gamma$ meant to generalize the ideal boundary of a Kleinian group. 
Its construction uses the auxiliary data of a {\bf Floyd function}, which is a function $f\colon \nats \to \real_{>0}$ satisfying
\begin{enumerate}[(i)]
\item $\sum_{n=1}^\infty f(n) < \infty$, and
\item there exists $m > 0$ such that $\frac1m \leq \frac{f(k+1)}{f(k)} \leq 1$ for all $k \in \nats$.
\end{enumerate}

Given such a function, there exists a metric $d_f$ on $\Gamma$ defined by setting $d_f(g,h) = f(\max\{|g|, |h|\})$ if $g,h$ are adjacent vertices in $\Cay(\Gamma,S)$, and considering the resulting path metric. Then the Floyd boundary $\del_f \Gamma$ with respect to $f$ is given by 
\[ \del_f \Gamma := \bar\Gamma \smallsetminus \Gamma \]
where $\bar\Gamma$ is the metric completion of $\Gamma$ with respect to the metric $d_f$.

Below, the Floyd boundary, in particular the ability of the Floyd function to serve as a sort of ``distance to infinity'', will be useful as a tool in the proof of Theorem \ref{thm:reldom_eiggap}. It may be possible, with more work, to replace the role of the Floyd boundary in that proof with the Bowditch boundary.

The Floyd boundary $\del_f \Gamma$ is called {\bf non-trivial} if it has at least three points. Gerasimov and Potyagailo have studied Floyd boundaries of relatively hyperbolic groups:
\begin{thm}[\cite{GerFloyd}, \cite{GP_FloydRH}] \label{thm:relhyp_floydbdy}
Suppose we have a non-elementary relatively hyperbolic group $\Gamma$ which is hyperbolic relative to $\mathcal{P}$.

Then there exists a Floyd function $f$ such that $\del_f \Gamma$ is non-trivial, and moreover 
\begin{enumerate}[(a)]
\item there exists a continuous equivariant map $F: \del_f G \to \del(\Gamma,\mathcal{P})$, such that
\item for any parabolic point $p \in \del(\Gamma, \mathcal{P})$, we have $F^{-1}(p) = \del_f(\Stab_\Gamma p)$, and if there exist $a \neq b$ such that $F(a) = F (b) = p$, then $p$ is parabolic. 
\end{enumerate} 
\end{thm}

\begin{rmk} \label{rmk:floyd_relhyp}
It is an open question whether every group with a non-trivial Floyd boundary is relatively hyperbolic --- see e.g.\ \cite{Ivan_thickFloyd}.
\end{rmk}

For more details, including justifications for some of the assertions above, we refer the reader to \cite{Floyd} and \cite{Karlsson}.

\subsection{Gromov products and translation lengths in hyperbolic spaces} \label{sub:stablen_grop}

We collect here, for the reader's convenience, assorted facts about Gromov products and translation lengths in Gromov-hyperbolic spaces that we use below, in particular in and around the statement and proof of Theorem \ref{thm:reldom_eiggap}.

Given $X$ a proper geodesic metric space, $x_0 \in X$ a fixed basepoint, and $\gamma$ an isometry of $X$, we define the {\bf translation length} of $\gamma$ as 
\[ \ell_X(\gamma) := \inf_{x \in X} d_X(\gamma x, x) \]
and the {\bf stable translation length} of $\gamma$ as 
\[ |\gamma|_{X,\infty} := \lim_{n\to\infty} \frac{d_X(\gamma^n x_0, x_0)}n .\]
When $X$ is $\delta$-hyperbolic space, these two quantities are coarsely equivalent:
\begin{prop}[{\cite[Chap.\ 10, Prop.\ 6.4]{CDP}}] \label{prop:Xhyp_translen_stranslen}
If $X$ is hyperbolic metric space, the quantities $\ell_X(\gamma)$ and $|\gamma|_{X,\infty}$ defined above satisfy
\[ \ell_X(\gamma) - 16\delta \leq |\gamma|_{X,\infty} \leq \ell_X(\gamma) .\]
\end{prop}

The {\bf Gromov product} with respect to $x_0$ is the function $\langle \cdot, \cdot \rangle_{x_0}\colon X \times X \to \real$ defined by
\[ \langle x, y \rangle_{x_0} := \frac12 \left( d_X(x,x_0) + d_X(y,x_0) - d_X(x,y) \right) .\]

There is a relation between the Gromov product, the stable translation length $|\gamma|_{X,\infty}$, and the quantity $|\gamma|_X = d_X(\gamma x_0, x_0)$, given by
\begin{lem} \label{lem:grop_len_stranlen}
Given $X$ a proper geodesic metric space, $x_0 \in X$ a basepoint, and $\gamma$ an isometry of $X$, we can find a sequence of integers $(m_i)_{i\in\nats}$
\[ 2 \lim_{i\to\infty} \langle \gamma^{m_i}, \gamma^{-1} \rangle_{x_0} \geq |\gamma|_X - |\gamma|_{X,\infty}.\]
\begin{proof}
By the definition of the stable translation length, we can find a sequence $(m_i)_{i\in\nats}$ such that 
\[ \lim_{i\to\infty} \left(|\gamma^{m_i+1}|_X - |\gamma^{m_i}|_X \right) \leq |\gamma|_{X,\infty} .\]

By the definition of the Gromov product,
\[ 2 \langle \gamma^{m_i}, \gamma^{-1} \rangle_{x_0} := |\gamma^{m_i}|_X + d_X(\gamma^{-1}x_0, x_0) - d_X(\gamma^{m_i} x_0,\gamma^{-1} x_0) .\]
Since $\gamma$ acts isometrically on $X$, $d_X(\gamma^{m_i} x_0,\gamma^{-1} x_0) = |\gamma^{m-i+1}|_X$ and $d_X(\gamma^{-1}x_0, x_0) = |\gamma|_X$. Then we have
\[ 2 \langle \gamma^{m_i}, \gamma^{-1} \rangle_{x_0} = |\gamma^{m_i}|_X + |\gamma|_X - |\gamma^{m_i+1}|_X \leq |\gamma|_X - |\gamma|_{X,\infty} \]
as desired.
\end{proof}
\end{lem}

\subsection{Singular value decompositions} \label{sub:SVD}

We collect here facts about singular values and Cartan decomposition in $\SL(d,\real)$. The defining conditions for our representations will be phrased, in the first instance, in terms of these, and more generally they will be helpful for understanding the geometry associated to our representations. 

Given a matrix $g \in \GL(d,\real)$, let $\sigma_i(g)$ (for $1 \leq i \leq d$) denote its $i$\textsuperscript{th} singular value, and write $U_i(g)$ to denote the span of the $i$ largest axes in the image of the unit sphere in $\real^d$ under $g$, and $S_i(g) := U_i(g^{-1})$. Note $U_i(g)$ is well-defined if and only if we have a singular-value gap $\sigma_i(g) > \sigma_{i+1}(g)$.

More algebraically, given $g \in \GL(d,\real)$, we may write $g = KAL$, where $K$ and $L$ are orthogonal matrices and $A$ is a diagonal matrix with nonincreasing positive entries down the diagonal. The diagonal matrix $A$ is uniquely determined, and we may define $\sigma_i(g) = A_{ii}$; $U_i(g)$ is given by the span of the first $i$ columns of $K$.

For $g \in \SL(d,\real)$, this singular-value decomposition is a concrete manifestation of a more general Lie-theoretic object, a (particular choice of) Cartan decomposition $\SL(d,\real) = \SO(d) \cdot \exp(\mathfrak{a}^+) \cdot \SO(d)$, where $\SO(d)$ is the maximal compact subgroup of $\SL(d,\real)$, and $\mathfrak{a}^+$ is a positive Weyl chamber.

We recall that there is an adjoint action $\Ad$ of $\SL(d,\real)$ on $\mathfrak{sl}(d,\real)$. 

We will occasionally write (given $g = KAL$ as above) 
\[ a(g) := (\log A_{11}, \dots, \log A_{dd}) = (\log \sigma_1(g), \dots, \log \sigma_d(g)) ;\] 
we note that the norm $\|a(g)\| = \sqrt{(\log \sigma_1(g))^2 + \dots + (\log \sigma_d(g))^2}$ is equal to the distance $d(o, g \cdot o)$ in the associated symmetric space $\SL(d,\real)/\SO(d)$ (see e.g.\ formula (7.3) in \cite{BPS}.)

\subsection{Regular ideal points and the projective space} \label{sub:regular_ideal}
Finally, we collect here some remarks about a subset of the visual boundary which will be relevant to us, and its relation to the projective space as a flag space boundary.

Given fixed constants $C_r, c_r > 0$, a matrix $g \in \SL(d,\real)$ will be called {\bf $(P_1,C_r,c_r)$-regular} if it satisfies 
\begin{align}
\log \frac{\sigma_1}{\sigma_2}(g) \geq C_r \log\frac{\sigma_1}{\sigma_d}(g) - c_r .
\label{ineq:unif_reg}
\end{align}

Recall that the visual boundary of the symmetric space $\SL(d,\real) / \SO(d)$ consists of equivalence classes of geodesic rays, where two rays are equivalent if they remain bounded distance apart. For any complete simply-connected non-positively curved Riemannian manifold $X$, such as our symmetric space, the visual boundary is homeomorphic to a sphere, and may be identified with the unit sphere around any basepoint $o$ by taking geodesic rays $\xi\colon [0,\infty) \to X$ based at $o$ and identifying $\xi(1)$ on the unit sphere with $\lim_{t\to\infty} \xi(t)$ in the visual boundary. 

The set of all points in this visual boundary which are accumulation points of sequences $(B_n \cdot o)$, where $o$ varies over all possible basepoints in the symmetric space and $(B_n)$ over all divergent sequences of $(P_1,C_r,c_r)$-regular matrices with all $c_r > 0$, will be called the {\bf $(P_1,C_r)$-regular ideal points}. 

For fixed $C_r$, the set of $(P_1,C_r)$-regular ideal points is compact; indeed it has the structure of a fiber bundle over the projective space $\proj(\real^d)$ with compact fibers.

There is a map $\pi$ (a fibration) from the set of $(P_1,C_r)$-regular ideal points to $\proj(\real^d)$ given by taking $\lim_n g_n \cdot o$ to $\lim_{n\to\infty} U_1(g_n)$ 
(see \cite[\S\S2.5.1 \& 4.6]{KLP}, where this is stated in slightly different language, or \cite[Th.\,7.2]{reldomreps}).
The map $\pi$ is Lipschitz, with Lipschitz constant depending only on the regularity constant $C_r$ and the choice of basepoint $o$ implicit in the measurement of the singular values 
\cite[\S4.4]{Max}. 

\section{Relatively dominated representations} \label{sec:D+-EC}

\begin{defn}[{\cite[\S4]{reldomreps}}] \label{defn:reldomrep}
Let $\Gamma$ be a finitely-generated torsion-free group which is hyperbolic relative to a collection  $\mathcal{P}$ of proper infinite subgroups.

Let $S$ be a compatible generating set, and let $X = X(\Gamma, \mathcal{P}, S)$ be the corresponding cusped space
(see Definitions \ref{defn:combhoroball} and \ref{defn:cuspedspace} above.) As above, let $d_c$ denote the metric on $X$, and $|\cdot|_c := d_c(\id, \cdot)$ denote the cusped word-length.

A representation $\rho\colon\Gamma \to \GL(d,\real)$ is {\bf $P_1$-dominated relative to $\mathcal{P}$}, with lower domination constants $\ubar{C}, \ubar{\mu} > 0$, if it satisfies \begin{itemize}
\item (D$-$) for all $\gamma \in \Gamma$, $\frac{\sigma_1}{\sigma_{2}}(\rho(\gamma)) \geq \ubar{C} e^{\ubar\mu|\gamma|_c}$, 
\end{itemize} 
and the images of peripheral subgroups under $\rho$ are well-behaved, meaning that the following three conditions are satisfied:
\begin{itemize}
\item (D+) there exist constants $\bar{C}, \bar\mu > 0$ such that $\frac{\sigma_1}{\sigma_d}(\rho(\eta)) \leq \bar{C} e^{\bar\mu|\eta|_c}$ 
for every $\gamma \in \Gamma$; 
\item (unique limits) for each $P \in \mathcal{P}$, there exists $\xi_\rho(P) \in \proj(\real^d)$ and  $\xi^*_\rho(P) \in \Gr_{d-1}(\real^d)$ such that for every sequence $(\eta_n) \subset P$ with $\eta_n \to \infty$, we have $\lim_{n\to\infty} U_1(\rho(\eta_n)) = \xi_\rho(P)$ and $\lim_{n\to\infty} U_{d-1}(\rho(\eta_n)) = \xi^*_\rho(P)$;
\item (uniform transversality)
for every $P, P' \in \mathcal{P}$ and $\gamma \in \Gamma$, $\xi_\rho(P) \neq \xi_\rho(\gamma P'\gamma^{-1})$. 
Moreover,
for every $\ubar\upsilon,\bar\upsilon>0$, 
there exists $\delta_0 > 0$ 
such that for all $P, P' \in \mathcal{P}$ and $g, h \in \Gamma$
such that there exists a bi-infinite $(\ubar\upsilon,\bar\upsilon)$-metric quasigeodesic path $\eta gh \eta'$ where 
$\eta'$ is in $P'$ and $\eta$ is in $P$, 
we have 
\[ \sin \angle (g^{-1} \xi_\rho(P), h\, \xi^*_\rho(P')) > \delta_0 .\]
\end{itemize}
\end{defn}

\begin{rmk}
Since $\Gamma$ is finitely-generated, so are its peripheral subgroups, by \cite[Prop.\,4.28 \& Cor.\,4.32]{DGO}.
\end{rmk}

\begin{rmk}
It is also possible to formulate the definition without assuming relative hyperbolicity, if one imposes additional hypotheses (RH) (see below) on the peripheral subgroups $\mathcal{P}$; it is then possible to show that any group admitting such a representation must be hyperbolic relative to $\mathcal{P}$: see \cite{reldomreps} for details.
\end{rmk}

The definition which originally appeared in \cite{reldomreps} also had an additional  ``quadratic gaps'' hypothesis, as part of the definition of the peripheral subgroups having well-behaved images. The only input of this assumption into the subsequent results there was in \cite[Lem.\,5.4]{reldomreps}; the next proposition obtains the conclusion of that lemma from the other hypotheses (not including relative hyperbolicity), without using the quadratic gaps hypothesis.

\begin{defn}[{\cite[Def.\ 4.1]{reldomreps}}] \label{defn:(RH)}
Given $\Gamma$ a finitely-generated group, we say that a collection $\mathcal{P}$ of finitely-generated subgroups satisfies (RH) if 
\begin{itemize}
\item (malnormality) $\mathcal{P}$ is malnormal, i.e.\ for all $\gamma \in \Gamma$ and $P, P' \in \mathcal{P}$, $\gamma P \gamma^{-1} \cap P' = 1$ unless $\gamma \in P = P'$; 
\item (non-distortion) there exists $\nu > 0$ such that for any infinite-order non-peripheral element $\gamma \in \Gamma$, $|\gamma^n|_c \geq \nu |n|$; 
\item (local-to-global) a sufficient long peripheral word $p'$ with sufficiently long overlap with a geodesic word $\gamma p$ combine to form a uniform quasigeodesic $\gamma p'$ (we refer the reader to \cite{reldomreps} for the precise formulation.)
\end{itemize}
\end{defn}
All of these conditions hold when $\Gamma$ is hyperbolic relative to $\mathcal{P}$ (see e.g.\ \cite{osinRH}).

\begin{defn}
Let $\alpha\colon\ints \to \Cay(\Gamma)$ be a bi-infinite path with $\alpha(\ints) \subset \Gamma$.

We define the sequence
\begin{align*} 
x_\gamma  & = ( \dots A_{a-1}, \dots, A_{-1}, A_0, \dots, A_{b-1}, \dots) \\
 & := \resizebox{.92\hsize}{!}{$(\dots, \rho(\alpha(a)^{-1} \alpha(a-1)), \dots, \rho(\alpha(0)^{-1} \alpha(-1)), \rho(\alpha(1)^{-1} \alpha(0)), \dots, \rho(\alpha(b)^{-1} \alpha(b-1)), \dots )$} 
 \end{align*}
and call this the {\bf matrix sequence associated to $\alpha$}.
\end{defn}

\begin{prop} \label{prop:D+-_EC}
Given a representation $\rho\colon(\Gamma,\mathcal{P}) \to \SL(d,\real)$ satisfying (D$\pm$) (so that $\mathcal{P}$ implicitly satisfies (RH), and we can define a cusped space $X(\Gamma,\mathcal{P})$), 
and given $\ubar\upsilon, \bar\upsilon > 0$, there exist constants $C \geq 1$ and $\mu >0$, depending only on the representation $\rho$ and $\ubar\upsilon, \bar\upsilon$, such that for any matrix sequence $x = x_{\gamma}$ associated to a 
bi-infinite $(\ubar\upsilon,\bar\upsilon)$-metric quasigeodesic path $\gamma$ with $\gamma(0) = \id$,
\begin{align*}
d(U_1(A_{k-1} \cdots A_{k-n}), U_1(A_{k-1} \cdots A_{k-(n+1)})) & \leq C e^{-n\mu} \\
d(S_{d-1}(A_{k+n-1} \cdots A_k), S_{d-1}(A_{k+n} \cdots A_k)) & \leq C e^{-n\mu} .
\end{align*}
\end{prop}

\begin{proof}
Given (D$\pm$), there exists $C_r,c_r > 0$ such that inequality (\ref{ineq:unif_reg})
is satisfied 
for all $\gamma \in \Gamma$.
Specifically, we can take $C_r = {\ubar\mu}/{\bar\mu}$ and $c_r = ({\ubar\mu}/{\bar\mu}) \log \bar{C} - \log \ubar{C}$, where $\ubar{C},\ubar\mu,\bar{C},\bar\mu$ are the constants coming from the (D$\pm$) conditions.
In the language of Kapovich--Leeb--Porti --- see \cite{KLP}, or \cite{KL} for the relative case; we adapt the relevant parts of this language and framework here --- $\rho(\Gamma)$ is a uniformly regular subgroup of $\SL(d,\real)$.

Hence $\rho(\gamma)$ is $(P_1,C_r,c_r)$-regular, in the sense of \S\ref{sub:regular_ideal}, for all $\gamma \in \Gamma$, and given a divergent sequence $(\gamma_n)$, $\rho(\gamma_n) \cdot o$ converges to a $(P_1,C_r,c_r)$-regular ideal point in the visual boundary.

Roughly speaking, geodesics converging to $(P_1,C_r,c_r)$-regular ideal points have as many hyperbolic directions as possible in the symmetric space, and thus flag convergence along these geodesics should occur exponentially quickly, just as in the hyperbolic case. This intuition can be made precise with more work, as follows:

Recall that we have a Lipschitz map $\pi$ from the set of $(P_1,C_r)$-regular ideal points to $\proj(\real^d)$, with Lipschitz constant depending only on the regularity constant $C_r$ and the choice of basepoint $o$ implicit in the measurement of the singular values.

Moreover, since $\rho(\gamma)$ is $(P_1,C_r,c_r)$-regular for any $\gamma \in \Gamma$, given the Cartan decomposition $\rho(\gamma) = K_\gamma \cdot \exp(a(\rho(\gamma))) \cdot L_\gamma$, we have 
\[ \Xi_\rho(\gamma) = \pi \left( \lim_{n\to\infty} K_\gamma \cdot \exp(n a(\rho(\gamma))) \cdot L_\gamma \cdot o \right) .\]

Thus, given any sequence  $(\gamma_n) \subset \Gamma$, we have 
\[ d(\Xi_\rho(\gamma_n), \Xi_\rho(\gamma_m)) \leq C_{Lip} \cdot \sin \angle \left( \Ad(K_{\gamma_n})\cdot a(\rho(\gamma_n)), \Ad(K_{\gamma_m})\cdot a(\rho(\gamma_m)) \right) .\]

Now, if $x = x_\gamma = (A_n)_{n\in\nats}$ is a matrix sequence associated to a bi-infinite $(\ubar\upsilon,\bar\upsilon)$-metric quasigeodesic path $\gamma$ with $\gamma(0) = \id$, then $A_{k-1} \dots A_{k-n} = \rho(\gamma(k)^{-1} \gamma(k-n))$. We write $\rho(\gamma(k)^{-1} \gamma(k-n)) = K_{k,n} \cdot \exp(a(k,n)) \cdot L_{k,n}$ to denote the parts of the Cartan decomposition.

By $(P_1,C_r,c_r)$-regularity and the higher-rank Morse lemma \cite[Th.\,1.1]{KLP_Morse}, the limit 
$$\lim_{n\to\infty} U_1(A_{k-1} \cdots A_{k-n}) = \lim_{n\to\infty} K_{k,n} \langle e_1 \rangle = \lim_{n\to\infty} \Ad(K_{k,n})\cdot a(k,n)$$ exists, and 
we have a bound $C_a$ on the distance\footnote{For readers more acquainted with the language of Kapovich--Leeb--Porti: this is the distance to the Weyl cone over the $C_r$-regular open star of $\lim_n K_{k,n} \langle e_1 \rangle$.} from $A_{k-1} \cdots A_{k-n} \cdot o$ to a nearest point on any $(P_1,C_r)$-regular ray $(g_n \cdot o)$ starting at $o$ such that $\lim_{n\to\infty} U_1(g_n) = \lim_n K_{k,n} \langle e_1 \rangle$ (below, we refer to any such point as $\pi_{\lim} A_{k-1} \cdots A_{k-n} \cdot o$), where $C_a$ depends only on $C_r, c_r$ and $\ubar\upsilon,\bar\upsilon$. 

Then, by \cite[Lem.\,4.9]{Max} applied with $p=o$ our basepoint, $\alpha_0 = C_r$, $\tau$ a model Weyl chamber corresponding to the first singular value gap, $q = A_{k-1} \cdots A_{k-n} \cdot o$, the point $r = \pi_{\lim}\, q$, the constant $2l = \|a(k,n)\| \geq \ubar\upsilon^{-1} n - \ubar\upsilon$ and $D = C_a$, we have
\begin{align*}
\sin \angle \left( \Ad(K_{k,n}) \, a(k,n), \lim_n K_{k,n} \langle e_1 \rangle \right) 
 & = \sin \angle \left( \frac12 \Ad(K_{k,n})\, a(k,n), \lim_n K_{k,n} \langle e_1 \rangle \right) \\ 
 & \leq \frac{d(q/2,\pi_{\lim}q/2)}{d(o,\pi_{\lim} q/2)} \\
 & \leq \frac{2C_a e^{C_a / \sqrt{d} + \ubar\upsilon/2} e^{-(C_r / 2\ubar\upsilon) n}}{d(o,\pi_{\lim}\, q/2)} \\
 & \leq 2C_a e^{C_a / \sqrt{d} + \ubar\upsilon/2} e^{-(C_r / 2\ubar\upsilon) n}
\end{align*}
once $n$ is sufficiently large, where ``sufficiently large'' depends only on the dimension $d$, our constants $C_r, C_a$ and choice of basepoint $o$; here $q/2$ denotes the midpoint of $oq$, which can be written as 
\[ K_{k,n} \cdot \exp\left(\frac12 a(k,n)\right) \cdot L_{k,n} \cdot o . \]

Hence we can find $\hat{C} \geq 2C_a e^{C_a / \sqrt{d} + \ubar\upsilon/2}$ such that 
$$ \sin \angle \left( \Ad(K_{k,n}) \, a(k,n), \lim_n K_{k,n} \langle e_1 \rangle \right) 
 \leq \hat{C} e^{-(C_r / 2\ubar\upsilon) n} $$
for all $n$, and so 
$d\left(U_1(A_{k-1} \cdots A_{k-n}), U_1(A_{k-1} \cdots A_{k-n-1}) \right)$ is bounded above by
\begin{align*}
 & C_{Lip} \sin \angle \left( \Ad(K_{k,n}) \, a(k,n), \Ad(K_{k,n+1}) \, a(k,n+1) \right) \\
 & \quad \leq C_{Lip} \hat{C} \left( 1+ e^{-C_r/2\ubar\upsilon} \right) e^{-(C_r/2\ubar\upsilon) n} .
\end{align*}
This gives us the desired bound with 
\begin{align*}
\mu = \frac12 C_r \ubar\upsilon^{-1} = \frac12 \ubar\mu (\bar\mu \ubar\upsilon)^{-1} && \text{and} && C = C_{Lip} \hat{C} \left( 1+ e^{-\mu} \right) .
\end{align*}

The analogous bound for $d\left(S_{d-1}(A_{k+n-1} \cdots A_k), S_{d-1}(A_{k+n} \cdots A_k) \right)$ can be obtained by arguing similarly, or by working with the dual representation --- for the details of this part we refer the interested reader to the end of the proof of \cite[Lem.\ 5.4]{reldomreps}.
\end{proof}

\section{A characterisation using eigenvalue gaps} \label{sec:reldom_eiggap}

Suppose $\Gamma$ is hyperbolic relative to $\mathcal{P}$. We have, as above, the cusped space $X = X(\Gamma,\mathcal{P})$, which is a $\delta$-hyperbolic space on which $\Gamma$ acts isometrically and properly. We define $|\cdot|_{c,\infty}$ to be the stable translation length on this space, i.e.
\[ |\gamma|_{c,\infty} := \lim_{n\to\infty} \frac{|\gamma^n|_c}n\] 
where $|\cdot|_c := d_X(\id,\cdot)$ as above. 

We remark that by Proposition \ref{prop:Xhyp_translen_stranslen} the eigenvalue gap conditions below may be equivalently formulated in terms of the translation length $\ell_X(\gamma)$.  

Given $A \in \GL(d,\real)$, let $\lambda_i(A)$ denote the magnitude of the $i$\textsuperscript{th} largest eigenvalue of $A$. 
We will prove the following theorem.
We remind the reader that the (D$\pm$) and (D$^\lambda_\pm$) conditions referred to in the theorem statement were defined in Definition \ref{defn:reldomrep} and in the statements of Theorems \ref{thm:intro_reldom_eiggap} and \ref{thm:intro_gaps+maps}. The limit set $\Lambda_1(\rho)$ which appears in the theorem statement was also defined in the statement of Theorem \ref{thm:intro_gaps+maps}.

\begin{thm} \label{thm:reldom_eiggap}
Let $\Gamma$ be hyperbolic relative to $\mathcal{P}$ and
$\rho\colon \Gamma \to \SL(d,\real)$ be a representation.
If $\rho$ satisfies (D$\pm$), then it satisfies (D$^\lambda_\pm$).

Conversely, if $\rho$ satisfies (D$^\lambda_\pm$), and either (a) is semisimple or (b) admits 
continuous, $\rho$-equivariant, transverse limit maps
$(\xi,\xi^*) \colon \del(\Gamma,\mathcal{P}) \to \proj(\real^d) \times \proj(\real^{d*})$
such that $\xi$ is a homeomorphism onto $\Lambda_1(\rho)$, then $\rho$ also satisfies (D$\pm$).
\end{thm}

\begin{cor} [Theorem \ref{thm:intro_reldom_eiggap}] \label{cor:reldom_eiggap} 
Let $\Gamma$ be hyperbolic relative to $\mathcal{P}$.
A semisimple representation $\rho\colon \Gamma \to \SL(d,\real)$ is $P_1$-dominated relative to $\mathcal{P}$ if and only if it satisfies ($\mathrm{D}^\lambda_\pm$) as well as the unique limits and uniform transversality conditions from Definition \ref{defn:reldomrep}.
\end{cor}

\begin{rmk}
We remark that the ($\mathrm{D}^\lambda_+$) condition really is equivalent to requiring if $\eta$ is a peripheral element (so $|\eta|_{c,\infty} = 0$), then all the eigenvalue of $\rho(\eta)$ have absolute value 1. If $\gamma$ is such that $|\gamma|_{c,\infty} = |\gamma|_\infty$, then the condition always holds because $\Gamma$ is finitely-generated; more generally we have 
\begin{align*}
\lambda_1(\rho(\gamma\eta)) & \leq \lambda_1(\rho(\gamma)) \lambda_1(\rho(\eta)),\mbox{ and} \\ 
\lambda_d(\rho(\gamma\eta)) &= \lambda_1(\rho(\eta^{-1} \gamma^{-1}) )^{-1} \geq \lambda_d(\rho(\gamma)) \lambda_d(\rho(\eta)), 
\end{align*}
and we can use these to piece together the condition on non-peripheral and peripheral parts of the word for $\gamma$.
\end{rmk}

\begin{proof}[Proof of Theorem \ref{thm:reldom_eiggap}]
We recall the identity $\log \lambda_i(A) = \lim_{n\to\infty} \frac{\log \sigma(A^n)}n$. Given (D$-$), we have
\begin{align*}
(\log \lambda_1 - \log \lambda_2) (\rho(\gamma)) & = \lim_{n\to\infty} \frac1n (\log \sigma_1 - \log \sigma_2) (\rho(\gamma^n)) \\ 
 & \geq \lim_{n\to\infty} \frac1n (\log \ubar{C} + \ubar\mu |\gamma^n|_c) = \ubar\mu |\gamma|_{c,\infty} 
\end{align*}
and so
\begin{align*}
\frac{\lambda_1}{\lambda_2}(\rho(\gamma)) & \geq e^{\ubar\mu |\gamma|_{c,\infty}}.
\end{align*}
Given (D+), we have
\begin{align*}
(\log \lambda_1 - \log \lambda_d) (\rho(\gamma)) & = \lim_{n\to\infty} \frac1n (\log \sigma_1 - \log \sigma_d) (\rho(\gamma^n)) \\ 
 & \leq \lim_{n\to\infty} \frac1n (\log \bar{C} + \bar\mu |\gamma^n|_c) = \bar\mu |\gamma|_{c,\infty}
\end{align*}
and so
\begin{align*}
\frac{\lambda_1}{\lambda_d}(\rho(\gamma)) & \leq e^{\bar\mu |\gamma|_{c,\infty}}.
\end{align*}

Hence (D$\pm$) implies (D${}^\lambda_\pm$).

In the other direction, we remark that by \cite[Th.\,5.3]{Kostas_AnosovWEG} together with Theorem \ref{thm:relhyp_floydbdy},
$\Gamma$ satisfies property U, i.e.\ there exist a finite subset $F \subset \Gamma$ and a constant $L > 0$ such that for every $\gamma \in \Gamma$ there exists $f \in F$ with
\begin{equation} \label{eq:propweakU}
|f\gamma|_\infty \geq |f\gamma| - L .
\end{equation}

We observe that this means that given any $\gamma \in \Gamma$ and $\eps > 0$, there exists $n_0 >0$ such that $|(f\gamma)^n| \geq n|f\gamma| - (1-\eps)Ln$ for all $n \geq n_0$, or in words there is bounded cancellation between the start and end of $f\gamma$.

We will now leverage this to obtain a relative version of the previous inequality, namely that for any given $\eps >0$, there exists $n_1 > 0$ such that $$|(f\gamma)^n|_c \geq \frac1{12} |f\gamma|_c - L$$ for all $n \geq n_1$.

To do so, we will impose some additional requirements on the finite set $F$ appearing above.

To describe these requirements, and to prove our relative inequality, we will use the framework and terminology described in \S\ref{sub:hat_unhat}. Abusing notation slightly, write $f\gamma$ to denote a geodesic path from $\id$ to $f\gamma$ in the Cayley graph. Consider this $f\gamma$ as a relative path $(f\gamma,H)$ with $H = H_1 \cup \dots \cup H_k$, and write $\eta_i = f\gamma|_{H_i}$, so each $\eta_i$ is a peripheral excursion. 

\begin{lem} \label{lem:adapt_T53}
Given $\Gamma$ a non-elementary relatively hyperbolic group, there exists a finite subset $F \in \Gamma$ and a constant $L > 0$ such that for every $\gamma \in \Gamma$
there exists $f \in F$ such that
\[ |f\gamma|_\infty \geq |f\gamma| - L \]
and the peripheral excursions of $(f\gamma)^n$ are precisely $n$ copies of the peripheral excursions of $f\gamma$.
\end{lem}

We defer the proof of this statement and first complete the proof of the theorem given the statement. 

By Proposition \ref{prop:unhat_distance},
\[ |f\gamma|_c \leq 4 \left( \ell(f\gamma) - \sum_{i=1}^k\ell(\eta_i) + \sum_{i=1}^k \hat\ell(\eta_i) \right).\]
By (\ref{eq:propweakU}), $\ell((f\gamma)^n) \geq n|f\gamma| - (1-\eps)Ln$ for all sufficiently large $n$. 
Crucially, by our assumption on the peripheral excursions of $(f\gamma)^n$, the total length of peripheral excursions for $(f\gamma)^n$ remains $n \sum_{i=1}^k \ell(\eta_i)$, and
the sum of the resulting $\hat\ell$ remains $n \sum_{i=1}^k \hat\ell(\eta_i)$.

Now we may use Proposition \ref{prop:unhat_distance} to conclude that 
\[ |(f\gamma)^n|_c \geq \frac13 \left(n\ell(f\gamma) - n\sum_{i=1}^k \ell(\eta_i) + n \sum_{i=1}^k \hat\ell(\eta_i) - (1-\eps)Ln \right) .\]

But this implies
\begin{align*}
|f\gamma|_{c,\infty} & = \lim_{n\to\infty} \frac1n |(f\gamma)^n|_c \\
 & \geq \frac13 \left( \ell(f\gamma) - \sum_{i=1}^k \ell(\eta_i) + \sum_{i=1}^k \hat\ell(\eta_i)  - L \right) \\
 & > \frac1{12} |f\gamma|_c - L.
\end{align*}
as desired. 

On the other hand it is clear that $|f\gamma|_{c,\infty} \leq |f\gamma|_c$.

Now, there exists a finite $F' \subset \Gamma$ and $C > 0$ such that for every $\gamma \in \Gamma$ there exists $f'\in F'$ such that for every $i$,
\[ |\log \lambda_i(\rho(\gamma f')) - \log \sigma_i(\rho(\gamma)) | \leq C .\]
This follows (a) for semisimple $\rho$ from \cite[Th.\,2.6]{Kostas_AnosovWEG}, and (b) in the case of $\rho$ admitting equivariant transverse homeomorphisms onto $\Lambda_1(\rho)$ from the following two lemmas (directly for $i=1$, and via suitable exterior powers for more general $i$). 

\begin{lem}
Suppose that $X$ is a proper geodesic hyperbolic metric space, $d_\infty$ is a visual metric on $\del_\infty X$,  and $\Gamma \leq \Isom(X)$ is a non-elementary discrete subgroup. There exist $\eps > 0$ and a finite set $F \subset \Gamma$ such that: if $\gamma \in \Gamma$, then there exists some $f \in F$ such that $\gamma f$ is hyperbolic and $d_\infty( (\gamma f)^+, (\gamma f)^-) > \eps$. 
\end{lem}

This lemma is likely well-known; for a proof, see e.g.\ \cite[Lem.\ A.2]{ZZ}.

\begin{prop} \label{prop:singval_eig_bd}
Suppose $\Gamma$ is hyperbolic relative to $\mathcal{P}$ and $\rho \colon \Gamma \to \SL(d,\real)$ is a representation which (i) satisfies (D$^\lambda_-$), and (ii) admits continuous, $\rho$-equivariant, transverse, dynamics-preserving limit maps
$(\xi,\xi^*) \colon \del(\Gamma,\mathcal{P}) \to \proj(\real^d) \times \proj(\real^{d*})$
such that $\xi$ is a homeomorphism onto $\Lambda_1(\rho)$.

Then, given any $\epsilon > 0$, there exists a constant $C > 0$ such that
$$
\left| \log \frac{\sigma_1}{\lambda_1}(\rho(\gamma)) \right| < C
$$ 
for all hyperbolic $\gamma \in \Gamma$ with $d_\infty(\gamma^+, \gamma^-) > \epsilon$.
\end{prop}
As with a previous technical lemma, we will defer the proof of the proposition and first complete the proof of the theorem. 

Now, given ($\mathrm{D}^\lambda_\pm$), we have
\begin{align*}
\frac{\sigma_1}{\sigma_d}(\rho(\gamma)) & \leq C^2 \cdot \frac{\lambda_1}{\lambda_d}(\rho(\gamma f)) \\
 & \leq C^2 e^{\bar\mu |\gamma f'|_{c,\infty}}
 \leq C^2 e^{\bar\mu |\gamma f'|_c} \\
 & \leq C^2 (C_{F'})^{\bar\mu} \cdot e^{\bar\mu|\gamma|_c}
\end{align*}
where $C_{F'} := \max_{f'\in F'} e^{|f'|_c}$ and so (D+) holds. Given ($\mathrm{D}^\lambda_-$), we have
\begin{align*}
\frac{\sigma_1}{\sigma_2}(\rho(\gamma)) & \geq C^{-2} \cdot \frac{\lambda_1}{\lambda_2}(\rho(\gamma f')) \\
 & \geq C^{-2} \ubar{C} e^{\ubar\mu |\gamma f'|_{c,\infty}}
 \geq C^{-2} \ubar{C} e^{-\ubar\mu L} e^{\frac1{12}\ubar\mu |f \gamma f'|_c} \\
 & \geq C^{-2} \ubar{C} e^{-\ubar\mu L} (C_F C_F')^{-\frac1{12}\ubar\mu} \cdot e^{\ubar\mu|\gamma|_c}
\end{align*}
where $C_{F'}$ is as above and $C_F := \max_{f\in F} e^{|f|_c}$, and hence (D$-$) holds.
\end{proof}

\begin{proof}[Proof of Lemma \ref{lem:adapt_T53}]
We adapt the proof of \cite[Th.\,5.3]{Kostas_AnosovWEG} to show that we can choose $F$ to satisfy the additional requirements we have imposed here. 

Let $f$ be a Floyd function $f\colon \nats \to \real^+$ for which the Floyd boundary $\del_f \Gamma$ of $\Gamma$ is non-trivial. By Theorem \ref{thm:relhyp_floydbdy}, there is a map from $\del_f\Gamma$ to the Bowditch boundary $\del(\Gamma,\mathcal{P})$ which is injective on the set of conical limit points; hence, by \cite[Prop.\,5]{Karlsson}, we can find {\it non-peripheral} $f_1, f_2$ such $\{f_1^+, f_1^-\} \cap \{f_2^+, f_2^-\} = \varnothing$. We will use sufficiently high powers of these to form our set $F$; the north-south dynamics of the convergence group action of $\Gamma$ on $\del_f \Gamma$ will do the rest.

To specify what ``sufficiently high'' means it will be useful to define an auxiliary function $G\colon \ints_{>0} \to \real_{>0}$, which gives a measure of ``distance to infinity'' as measured by the Floyd function: concretely, take $G(x) := 10 \sum_{k= \lfloor x/2 \rfloor}^\infty f(k)$. Since $f$ is a Floyd function, $G(x) \searrow 0$ as $x \to \infty$. By \cite[Lem.\,1]{Karlsson}\footnote{By the monotonicity and positivity of $f$ and because $x \in \ints_{>0}$, our choice of $G$ bounds from above the function $4xf(x) + 2 \sum_{k = x}^\infty f(k)$ appearing in Karlsson's proof.}, we have
\begin{align}
d_f(g,h) & \leq G\left( \langle g, h \rangle_e \right) & d_f(g,g^+) & \leq G\left( |g|/ 2 \right)
\label{eq:karlsson_est}
\end{align}
for all $g, h \in \Gamma$.
Let $\eps = \frac16 d_f(f_1^\pm, f_2^\pm)$. Fix $M>0$ such that $G(x) \geq \frac\eps{10}$ if and only if $x \leq M$, and $R>0$ such that $G(x) \leq \frac\eps{10}$ if and only if $x \geq R$, and $N$ such that $\min\{ |f_1^{N'}|, |f_2^{N'}|\} \geq 2(M+R)$ for all $N' \geq N$. 

\begin{claim} 
For every non-trivial $\gamma \in \Gamma$ such that $d_f(\gamma^+, \gamma^-) \leq \eps$, there exists $i \in \{1,2\}$ such that $d_f(f_i^{N'} \gamma^+, \gamma^-) \geq \eps$ for all $N' \geq N$. 
\begin{proof}[Proof of claim]
By our choice of $\eps$, we can find $i \in \{1,2\}$ such that $d_f(\gamma^+, f_i^+) \geq 3\eps$: if $d_f(\gamma^+, f_1^\pm) < 3\eps$, then $d_f(\gamma^+, f_2^\pm) \geq d(f_2^\pm, f_1^\pm) - 3\eps = 3\eps$. Without loss of generality suppose $i=1$. 

There exists $n_0$ such that $G\left( \frac12 |\gamma^n|\right) < \eps$ for all $n \geq n_0$. For $n \geq n_0$ and $N' \geq N$, by our choice of $N$, we have
\begin{align*}
d_f(\gamma^n, f_1^{-N'}) & \geq d_f(\gamma^+, f_1^-) - d_f\left(f_1^-, f_1^{-N'}\right) - d_f(\gamma^+, \gamma^n) \\
& \geq 3\eps - G \left( \frac12 |f_1^{N'}| \right)  - G\left( \frac12|\gamma^n| \right) > \eps.
\end{align*}
Hence, for all $n \geq n_0$ and $N' \geq N$, we have $G(\langle \gamma^n, f_1^{-N'} \rangle_e) \geq d_f(\gamma^n, f_1^{-N'}) > \eps$, and $\langle \gamma^n, f_1^{-N'} \rangle_e \leq M$ by our choice of $M$.
Now choose a sequence $(k_i)_{i\in\nats}$ such that $|f_1^{k_i-N}| < |f_1^{k_i}|$ for all $i \in \nats$. For $n \geq n_0$ and $N' \geq N$, we have, by the definition of the Gromov product and the inequalities above,
\begin{align*}
2 \langle f_1^{N'}\gamma^n, f_1^{k_n} \rangle_e & = |f_1^{N'} \gamma^n| + |f_1^{k_n}| - |f_1^{N'-k_n} \gamma^n| \\
 & = |\gamma^n| + |f_1^{N'}| - 2\langle \gamma^n, f_1^{-N'} \rangle_e + |f_1^{k_n}| - |f_1^{N'-k_n} \gamma^n| \\
 & \geq |f_1^{N'}| - 2M + |f_1^{k_n}| - (|f_1^{N'-k_n} \gamma^n|-|\gamma^n|) \\ & \geq |f_1^{N'}| - 2M + |f_1^{k_n}| - |f_1^{N'-k_n}| \\
 & \geq |f_1^{N'}| - 2M \geq 2R.
\end{align*}
Then by our choice of $R$ we have 
\[ d_f(f_1^{N'} \gamma^+, f_1^+) \leq \lim_{i\to\infty} G( \langle f_1^{N'}\gamma^n, f_1^{k_n} \rangle_e ) \leq \eps/10 \]
whenever $n \geq n_0$ and $N' \geq N$; thus 
\[ d_f(f_1^{N'}\gamma^+, \gamma^-) \geq d_f(\gamma^+,f_1^+) - d_f(f_1^{N'}\gamma^+, f_1^+) - d_f(\gamma^+,\gamma^-) \geq \eps \]
whence the claim.
\end{proof} \end{claim}

Now, with $f_1, f_2$ and $N$ as above, fix $F_0 = \{f_1^N, f_1^{N+1}, f_2^N, f_2^{N+1}, e\}$. Then there exists $g \in F_0$ such that $d_f(g\gamma^+, \gamma^-) \geq \eps$: if $d_f(\gamma^+, \gamma^-) \geq \eps$, choose $g = e$. Otherwise, from the above argument, either $g=f_1^N$ or $g=f_2^N$ works, and then so does $g=f_1^{N+1}$ or $g=f_2^{N+1}$ respectively.

Next fix $L = 2 \max_{g \in F_0} |g| + 2R + 1$; we will show that the desired result holds with $F := F_0 \cup S$ and this $L$. 
Without loss of generality suppose $|\gamma| > L-1$; otherwise $|\gamma| - |\gamma|_\infty \leq L$ and we have our desired inequality with $g=e$. Otherwise choose $g \in F_0$ such that $d_f(g \gamma^+, \gamma^-) \geq \eps$. To use this to obtain an inequality between $|g\gamma|$ and $|g\gamma|_\infty$, we use Lemma \ref{lem:grop_len_stranlen} with $g\gamma$ in the place of $\gamma$, the Cayley graph in the place of $X$, and $x_0 = e$ to obtain a sequence $(m_i)_{i\in\nats}$ such that
\begin{align}
\ 2 \lim_{i\to\infty} \langle (g\gamma)^{m_i}, (g\gamma)^{-1} \rangle_e \geq |g\gamma| - |g\gamma|_\infty , \label{eq:len_stablen_gropbd}
\end{align}
so it suffices to obtain an upper bound on the Gromov products $\langle (g\gamma)^{m_i}, (g\gamma)^{-1} \rangle_e$.

To obtain this bound, we start by noting that $g\gamma^+ = (g\gamma g^{-1})^+$, and using this, the triangle inequality, and the inequalities in (\ref{eq:karlsson_est}) to observe that
\begin{align*}
d_f\left( g\gamma^+, (g\gamma)^+ \right) & \leq d_f\left( g\gamma^+, g\gamma g^{-1}\right) + d_f\left(g\gamma g^{-1}, g\gamma\right) + d_f\left( (g\gamma)^+, g\gamma \right) \\
 & \leq G\left( \frac12 |g\gamma g^{-1}| \right) + G\left( \langle g\gamma g^{-1}, g\gamma \rangle_e \right) + G \left( \frac12|g\gamma| \right)
\end{align*}
and using liberally the monotonicity of $G$ on the last right-hand side, we obtain the further upper bound
\begin{align*}
d_f\left( g\gamma^+, (g\gamma)^+ \right) & \leq 3 G\left( \frac12|\gamma| - |g| \right)
\end{align*}
which, finally, because $\frac12 |\gamma| - |g| \geq R$, is bounded above by $\frac{3\eps}{10}$. 
Arguing similarly, we have
\begin{align*}
d_f\left( \gamma^{-1}, \gamma^{-1} g^{-1} \right) & \leq d_f\left(\gamma^-, \gamma^{-1} \right) + d_f \left( \gamma^{-1}, \gamma^{-1} g^{-1} \right) \\
 & \leq G \left( \frac12|g\gamma| \right) + G\left( \langle \gamma^{-1}, \gamma^{-1}g^{-1} \rangle_e \right) \\
 & \leq 2 G\left( \frac12|\gamma| - |g| \right) \leq \frac\eps5
\end{align*}
and hence we have
\begin{align*} 
d_f\left((g\gamma)^+, \gamma^{-1}g^{-1} \right) & \geq d_f\left( g\gamma^+, \gamma^- \right) - d_f\left( g\gamma^+, (g\gamma)^+ \right) - d_f\left( \gamma^-, \gamma^{-1}g^{-1} \right) \\ 
& \geq \eps - \frac{3\eps}{10} - \frac\eps 5 = \frac\eps2 .
\end{align*}
Thus we have $n_1 > 0$ such that $G \left( \langle (g\gamma)^n, (g\gamma)^{-1} \rangle_e \right) \geq d_f\left((g\gamma)^n, (g\gamma)^{-1} \right) \geq \frac\eps3$ and so $\langle (g\gamma)^n, (g\gamma)^{-1} \rangle_e \leq M$ for all $n \geq n_1$. This is the bound we feed into (\ref{eq:len_stablen_gropbd}) to obtain $|g\gamma| - |g\gamma|_\infty \leq 2M \leq 2R \leq L$, which was the inequality to be shown.

Finally, we prove the statement about  the peripheral excursions. 
We may also assume, without loss of generality, that $g\gamma$ contains at least one peripheral excursion, otherwise there is nothing left to prove. 

If we have a relation $\alpha\eta\beta$ with $\eta \in P \smallsetminus \{\id\}$ peripheral and $\alpha, \beta \notin P$ (and $\alpha$ not ending in any letter of $P$ and $\beta$ not starting in any letter of $P$), then $\alpha\eta\alpha^{-1} = \beta^{-1} \eta \beta$, and by malnormality this implies $\alpha = \beta^{-1}$, which is not possible since $\eta \neq \id$. Since we are assuming $g\gamma$ has peripheral excursions, we may thus assume that in $(g\gamma)^n$ there is no cancellation across more two copies of $g\gamma$, i.e.\ it suffices to look at cancellation between adjacent copies.

The peripheral excursions of $(g\gamma)^n$ are exactly $n$ copies of that of $g\gamma$ precisely when cancellation between adjacent copies of $g\gamma$ does not reach any of the peripheral excursions. 

Suppose now that this is not the case, i.e.\ cancellation between adjacent copies does reach the peripheral excursions. If $g = f_i^N$ (resp.\ $g=f_i^{N+1}$), then we may take $g = f_i^{N+1}$ (resp.\ $g=f_i^N$) instead; the desired inequalities still hold from the arguments above, and now cancellation between adjacent copies no longer reaches the peripheral excursions. 

Suppose instead $g = e$; then we may assume, from the argument above, that $|\gamma| \leq L-1$. We will instead take $g$ to be a non-peripheral generator $s$; then, while we had cancellation between adjacent copies before with $g=e$, we can no longer have it with $g=s$. Then $|s\gamma| \leq |\gamma| + 1 \leq L$, and we are done.
\end{proof}

\begin{proof}[Proof of Proposition \ref{prop:singval_eig_bd}]
Suppose for some $\eps>0$ no such $C>0$ exists, so that we have a sequence $(\gamma_n) \subset \Gamma$ of hyperbolic elements such that $d_\infty(\gamma_n^+, \gamma_n^-) > \eps$ for all $n$, and
$$
\left| \log \frac{\sigma_1}{\lambda_1}(\rho(\gamma_n f)) \right| \xrightarrow{n\to\infty} \infty .
$$ 


Let $\ell := \lim_n U_1(\rho(\gamma_n f))$. Given any biproximal $g \in \SL(d,\real)$, write $g^\pm$ to denote the attracting fixed line of $g^{\pm 1}$, i.e.\ the top eigenline, and $H_g^-$ to denote the repelling fixed hyperplane of $g$, i.e.\ the complementary subspace to $g^+$ preserved by $g$.

\begin{claim} 
$\ell$ is not transverse to $\lim_n H_{\rho(\gamma_n f)}^- = \xi^*(y)$.
\end{claim}
\begin{proof}[Proof of claim]
For each $n$, take a vector $v_n$ in $U_1(\rho(\gamma_n f))$ such that $u_n = \rho(\gamma_n f)^{-1} v_n$ is a unit vector, and write $v_n = v_{n,x} + v_{n,y}$ where $v_{n,y} \in \xi^*(y)$ and $v_{n,x} \in (\rho(\gamma_n f))^+$.
If $v_{n,x} = 0$ for all but finitely many $n$, then we already have that $\ell \not\pitchfork \xi^*(y)$; so (passing to a further subsequence if needed) suppose that $v_{n,x} \neq 0$ for all $n$.

Then, on the one hand,
$\|\rho(\gamma_n f) u_n\| = \sigma_1 (\rho(\gamma_n f))$ for all $n$. On the other hand,
\begin{align*}
\|\rho(\gamma_n f) u_n\| & = \|v_{n,x} + v_{n,y}\|
 \leq  \|v_{n,x}\| + \|v_{n,y}\| \\
 & = \|v_{n,x}\| \left( 1 + \frac{\|v_{n,y}\|}{\|v_{n,x}\|} \right)
  \leq \lambda_1(\rho(\gamma _n f)) \left( 1 + \frac{\|v_{n,y}\|}{\|v_{n,x}\|} \right). 
\end{align*}

If $\ell \pitchfork \xi^*(y)$, then the last ratio $\frac{\|v_{n,y}\|}{\|v_{n,x}\|}$ is bounded above by some $C'$ (for all sufficiently large $n$), and so, putting all of the above estimates together, we get
$$
\frac{\sigma_1}{\lambda_1}(\rho(\gamma_n f)) \leq 1+C'
$$
for all large enough $n$. Since we are assuming here that this doesn't happen, we conclude that $\ell \not\pitchfork \xi^*(y)$.
\end{proof}

Since $\xi$ is an equivariant homeomorphism onto $\Lambda_1(\rho)$, we have that $\ell = \xi(z)$ for some $z \in \partial(\Gamma,\mathcal{P})$.
By transversality, this implies that $\ell = \xi(y) = \lim_n (\rho(\gamma_n f))^{-}$. Now this will give us a contradiction since $U_1(\rho(\gamma_n f))$ is the line least expanded by $\rho(\gamma_n f)^{-1}$, 
but it cannot be so if it is so close to $\rho(\gamma_n f)^{-}$. More precisely, take a unit vector $w_n \in U_1(\rho(\gamma_n f))$. and write $w_n = w_{-,n} + w_{r,n}$, where $w_{-,n} \in \rho(\gamma_n f)^{1,-}$ and $w_{r,n} \perp w_{-,n}$. Then, given any $\epsilon' > 0$, for all large enough $n$, $\|w_{-,n}\| \geq 1-\epsilon'$, and
\begin{align*}
\| \rho(\gamma_n f)^{-1} w_n \| & \geq \lambda_1(\rho(\gamma_n f)^{-1}) (1 - \epsilon') \\
 & > \lambda_d(\rho(\gamma_n f)^{-1}) \geq \mu_d(\rho(\gamma_n f)^{-1}) ,
\end{align*}
where the middle inequality follows since 
\begin{align} \label{eq:singvalgap=>eiggap}
\log \frac{\lambda_1}{\lambda_d}(\rho(\gamma_nf))\geq \log \frac{\lambda_1}{\lambda_2}(\rho(\gamma_nf))
\xrightarrow{n\to\infty} \infty
\end{align}
by (D$^\lambda_-$) since there exists a uniform $\eps > 0$ such that $d_\infty(\gamma_n^+, \gamma_n^-) > \eps$ for all $n$.
\end{proof}

\section{Limit maps imply well-behaved peripherals} \label{sec:gaps+maps}

If we assume that our group $\Gamma$ is hyperbolic relative to $\mathcal{P}$, then the additional conditions of unique limits and uniform transversality which appear in either of the definitions of relatively dominated representations so far may also be replaced by a condition stipulating the existence of suitable limit maps from the Bowditch boundary $\del(\Gamma,\mathcal{P})$. As noted above, this gives us relative analogues of some of the characterizations of Anosov representations due to Gu\'eritaud--Guichard--Kassel--Wienhard \cite[Th.\,1.3 and 1.7 (1),(3)]{GGKW}.

\begin{thm}[Theorem \ref{thm:intro_gaps+maps}] \label{thm:gaps+maps}
Let $\Gamma$ be hyperbolic relative to $\mathcal{P}$.
A representation $\rho\colon \Gamma \to \SL(d,\real)$ is $P_1$-dominated relative to $\mathcal{P}$ if and only if (D$\pm$) (as in Definition \ref{defn:reldomrep} and Theorem \ref{thm:intro_gaps+maps}) are satisfied and there exist continuous, $\rho$-equivariant, transverse, dynamics-preserving limit maps 
$\xi_\rho\colon \del(\Gamma,\mathcal{P}) \to \proj(\real^d)$ and $\xi_\rho^*\colon \del(\Gamma,\mathcal{P}) \to \proj(\real^{d*})$. 
\begin{proof}
If $\rho$ is $P_1$-dominated relative to $\mathcal{P}$, then it satisfies (D$\pm$), and admits continuous, equivariant, transverse, dynamics-preserving limit maps \cite[Th.\,7.2]{reldomreps}.

Conversely, if suffices to show that the unique limits and uniform transversality conditions must hold once we have continuous, equivariant, transverse, dynamics-preserving limit maps. Unique limits follows from the limit maps being well-defined and dynamics-preserving, since there is a single limit point in $\del(\Gamma,\mathcal{P})$ for each peripheral subgroup. Transversality is immediate from the hypotheses, and the uniform version follows from a short argument, as done in \cite[Prop.\,8.5]{reldomreps}.
\end{proof}
\end{thm}

We remind the reader that the (D$^\lambda_\pm$) conditions and the limit set $\Lambda_1(\rho)$ which appear in the corollaries below were defined in the statement of Theorem \ref{thm:intro_gaps+maps}.
\begin{cor} \label{thm:eiggaps+maps_a}
Let $\Gamma$ be hyperbolic relative to $\mathcal{P}$.
A semisimple representation $\rho\colon \Gamma \to \SL(d,\real)$ is $P_1$-dominated relative to $\mathcal{P}$ if and only if (D$^\lambda_\pm$) are satisfied and there exist continuous, $\rho$-equivariant, transverse, dynamics-preserving limit maps
$(\xi_\rho,\xi_\rho^*) \colon \del(\Gamma,\mathcal{P}) \to \proj(\real^d) \times \proj(\real^{d*})$.
\begin{proof}
This follows immediately from Theorems \ref{thm:gaps+maps} and \ref{thm:reldom_eiggap}.
\end{proof}
\end{cor}

\begin{cor} \label{thm:eiggaps+maps_b}
Let $\Gamma$ be hyperbolic relative to $\mathcal{P}$.
A representation $\rho\colon \Gamma \to \SL(d,\real)$ is $P_1$-dominated relative to $\mathcal{P}$ if and only if (D$^\lambda_\pm$) are satisfied and there exist continuous, $\rho$-equivariant, transverse limit maps
$(\xi_\rho,\xi_\rho^*) \colon \del(\Gamma,\mathcal{P}) \to \proj(\real^d) \times \proj(\real^{d*})$
such that $\xi_\rho$ is a homeomorphism onto 
$\Lambda_1(\rho)$.
\begin{proof}
For a $P_1$-dominated representation, by \cite[Prop.\ 6.14 and Th.\ 7.2]{reldomreps}, the limit map $\xi$ is in fact an equivariant homeomorphism onto $\Lambda_1(\rho)$.

The desired result then follows from Theorems \ref{thm:gaps+maps} and \ref{thm:reldom_eiggap}.
\end{proof}
\end{cor}

As an application of Theorem \ref{thm:gaps+maps}, we can show that certain groups that play weak ping-pong on flag spaces are relatively dominated.
We remark that these examples have previously been claimed in \cite{KL}.

\begin{eg} \label{eg:proj_schottky}
Fix biproximal elements $t_1, \dots, t_k \in \PGL(d,\real)$. Write $t_i^\pm$ to denote the attracting lines and $H_{t_i}^\pm$ to denote the repelling hyperplanes of $t_i^{\pm1}$.

Assuming $t_i^+ \neq t_j^+$ for $i \neq j$ and $t_i^\pm \not\subset H_{t_j}^\mp$ for all $i, j$, and replacing the $t_i$ with sufficiently high powers if needed, we have open neighborhoods $A_i^\pm \subset \proj(\real^d) =: X$ of $t_i^\pm$, and $B_i^\pm \subset X$ of $H_{t_i}^\pm$ such that 
\begin{itemize}
    \item $A_i^\pm \subset B_i^\pm$ for $i=1\dots, k$, and $A_i^\sigma \cap B_j^{\sigma'} = \varnothing$ unless $i=j$ and $\sigma = \sigma'$,
    \item $t_i^{\pm1} \left( X \smallsetminus B_i^\pm \right) \subset A_i^\pm$ for $i=1, \dots, k$, and moreover
    \item there exists $\eps > 0$ such that $t_i^{\pm1}$ is $\eps$-Lipschitz on $X \smallsetminus B_i^\pm$ for all $i$ (see \cite[Lem.\ A.8]{CLSS}).
\end{itemize}

Suppose we have, in addition, unipotent elements $u_1, \dots, u_{k'} \in \PGL(d,\real)$ which each have well-defined attracting lines $u_j^+$
 and attracting hyperplanes $H_{u_j}^+$
(equivalently, well-defined largest Jordan blocks). Suppose, again passing to sufficiently high powers of the $u_1,\dots, u_{k'}$ if need be, there exist open neighborhoods $C_j^+$ of $u_j^+$ and $C_j^-$ of $H_{u_j}^+$ in $X = \proj(\real^d)$, such that 
\begin{itemize}
    \item $C_j^+ \subset C_j^-$ for $j=1, \dots, k'$, and the $\overline{C_1^+}, \dots, \overline{C_{k'}^+}$ are pairwise disjoint and also disjoint from the
the closures of all of the $B_i^\pm$, 
    \item $u_j^{\pm n}(X \smallsetminus C_j^-) \subset C_j^+$ for all non-zero $n$, and moreover
    \item there exists $c>0$ such that $u_j^{\pm n}$ is $\frac cn$-Lipschitz on $X \smallsetminus C_j^-$ for all $n  \in\ints_{>0}$.
\end{itemize}
To see that the we may assume the last hypothesis to hold: fix $u = u_j$. Let $v_1, \dots, v_d$ be a basis for $\real^d$ with respect to which $u$ may be written in Jordan normal form, where $v_1$ spans $u^+$ and $v_1, \dots, v_{d-1}$ span $H_{u}^+$. 

Up to introducing a biLipschitz error, we can choose a metric on $\proj(\real^d)$ given by pushing forward the suitable spherical metric obtained by viewing $u^+$ as the north pole and $\proj \langle v_2, \dots, v_d \rangle$ as the (projectivization of the) equator. 
In the affine chart given by taking $\langle v_2, \dots, v_d \rangle$ to be the hyperplane at infinity, 
if we consider polar coordinates $(r,\theta)$ with origin $u^+$, the spherical metric satisfies
\[ d\left( (r,\theta), (r',\theta') \right) \leq |\phi-\phi'| + \min\{\phi,\phi'\} \cdot |\theta-\theta'|\] 
where $\phi := \sin \arctan r$. 

Then, given two points $\xi_1, \xi_2 \in \proj(\real^d) \smallsetminus C_j^-$, with $\xi_i = (\theta_i,\phi_i)$ for $i=1,2$ in our coordinates, and abusing notation slightly to write $S^{\pm n} \ell_i = (S^{\pm n} \theta_i, S^{\pm n} \phi_i)$ for $i=1,2$, we have some constants $L, L' > 0$ such that
\begin{align*}
|S^{\pm n} \phi_2 - S^{\pm n} \phi_1| & \leq L\cdot\frac{\sigma_2}{\sigma_1}(S^{\pm n}) \cdot |\phi_2 - \phi_1| \leq \frac {L'}{n} \cdot |\phi_2 - \phi_1| \\
|S^{\pm n} \theta_2 - S^{\pm n} \theta_1| & \leq |\theta_2 - \theta_1|
\end{align*}
and so we have
\begin{align*}
d(S^{\pm n} \ell_1, S^{\pm n} \ell_2) & \leq \frac {L'}{n} \left( |\phi_2 - \phi_1| + \min( \phi_2, \phi_1) |\theta_2 - \theta_1| \right)
 \leq \frac{2L'}{n} \cdot d(\ell_1,\ell_2)
\end{align*} 
for all $n > 0$.
Hence we have the Lipschitz constants we seek.

Then, by a ping-pong argument, the group $\Gamma := \langle t_1, \dots, t_k, u_1, \dots, u_{k'} \rangle$ is isomorphic to a non-abelian free group $F_{k+k'}$.

Since we have finitely many generators, we can pick $\eps_0>0$ such that 
\begin{itemize}
    \item for all $i=1,\dots,k$ and for any $n > 0$ (resp.\ $n < 0$), $U_1(t_i^n)$ is within $\eps_0$ of $t_i^+$ (resp.\ $t_i^-$),
    \item for all $i=1,\dots,k$ and for any $n < 0$ (resp.\ $n > 0$), $U_{d-1}(t_i^n)$ is within $\eps_0$ of $H_{t_i}^+$ (resp.\ $H_{t_i}^-$),
    and
    \item for all $j=1,\dots,k'$ and for any $n \neq 0$, $U_1(u_j^{\pm n})$ are within $\eps_0$ of $u_j^+$.
\end{itemize}

By taking powers of the generators and slightly expanding the ping-pong neighborhoods if needed, we may assume that $\eps_0$ is sufficiently small so that the $A_i^\pm$ and $B_i^\pm$ contain the $2\eps_0$-neighborhoods of the $t_i^\pm$ and $H_{t_i}^\pm$ respectively, and the $C_j^+$ and $C_j^-$ contain the $2\eps_0$-neighborhoods of the $u_j^+$ and $H_{u_j}^+$ respectively. 
This slight strengthening of ping-pong will be useful for establishing the transversality of our limit maps below.

Below, we replace $\Gamma$ by the free subgroup generated by these powers.

Let $\mathcal{P} = \left\{\langle u_1 \rangle, \dots, \langle u_{k'} \rangle \right\}$.
Then $\Gamma$ is hyperbolic relative to $\mathcal{P}$ and there are continuous $\Gamma$-equivariant homeomorphisms $\xi, \xi^*$ from the Bowditch boundary $\del(\Gamma,\mathcal{P})$ 
to the limit set $\Lambda_{\Gamma} \subset \proj(\real^d)$ and the dual limit set $\Lambda_{\Gamma}^* \subset \proj(\real^{d*})$ given by 
\begin{align*}
\lim_n \gamma_n \mapsto \lim_n U_1(\gamma_n) && \text{and} && \lim_n \gamma_n \mapsto \lim_n U_{d-1}(\gamma_n) 
\end{align*}
respectively (cf.\ \cite[Prop.\ A.5]{CLSS}).

By definition, $\xi$ and $\xi^*$ are dynamics-preserving.

We claim that $\xi$ and $\xi^*$ are transverse: given two distinct points $x = \lim \gamma_n$ and $y = \lim \eta_n$ in $\del(\Gamma,\mathcal{P})$, we have $\xi(x) \notin \xi^*(y)$ --- the latter considered as a projective hyperplane in $\proj(\real^d)$ --- using ping-pong and the following 
\begin{lem}[{\cite[Lem.\,5.8]{GGKW}; \cite[Lem.\,A.5]{BPS}}]
If $A, B \in \GL(d,\real)$ are such that $\sigma_p(A)> \sigma_{p+1}(A)$ and $\sigma_p(AB)> \sigma_{p+1}(AB)$, then
\[ d\left(B\cdot U_p(A), U_p(BA) \right) \leq \frac{\sigma_1}{\sigma_d}(B) \cdot \frac{\sigma_{p+1}}{\sigma_p}(A) .\]
\end{lem}
To establish the claim: write $\gamma_n = g_1 \cdots g_n$ and $\eta_n = h_1 \dots h_n$. Pick $n_0$ minimal such that $U_1(\gamma_{n_0})$ and $U_1(\eta_{n_0})$ are in different ping-pong sets. The lemma above implies that for any given $\eps >0$, there exists some $n_1$ so that for all $n \geq n_1$, $U_1(\gamma_n) = U_1(g_1 \cdots g_n)$ is $\eps$-close to $\gamma_{n_0} \cdot U_1(g_{n_0+1} \cdots g_n)$, and $U_{d-1}(\eta_n)$ is $\eps$-close to $\eta_{n_0} \cdot U_{d-1}(h_{n_0+1} \cdots  h_n)$. By our ping-pong setup, for sufficiently small $\eps$ these are uniformly close to $U_1(\gamma_{n_0})$ and $U_{d-1}(\eta_{n_0})$ respectively, and in particular they are transverse to each other.

Moreover, the inclusion $\iota\colon \Gamma \into \PGL(d,\real)$ satisfies (D$\pm$); 
the proof of this claim will not require the strengthened version of ping-pong described above.

(D+) is immediate from $\Gamma$ being finitely-generated, the existence of a polynomial $\bar{q}$ of degree $d-1$ such that  $\frac{\sigma_1}{\sigma_d}(u) \leq \bar{q}(|u|)$ for every unipotent element $u \in \Gamma$, and the sub-multiplicativity of the first singular value $\sigma_1$.

To obtain (D$-$), one can use the following 
\begin{lem}[{\cite[Lem.\,A.7]{BPS}}] \label{lem:BPSA7}
If $A, B \in \GL(d,\real)$ are such that $\sigma_p(A)> \sigma_{p+1}(A)$ and $\sigma_p(AB)> \sigma_{p+1}(AB)$, then
\begin{align*}
\sigma_p(AB) & \geq (\sin\alpha) \cdot \sigma_p(A)\, \sigma_p(B) \\
\sigma_{p+1}(AB) & \leq (\sin\alpha)^{-1} \sigma_{p+1}(A)\, \sigma_{p+1}(B)
\end{align*}
where $\alpha := \angle \left( U_p(B), S_{d-p}(A)\right)$.
\end{lem}
To use this here, we show that there exists a uniform constant $\alpha_0 > 0$ such that whenever $(\gamma_n = g_1 \cdots g_n)_{n\in\nats} \subset \Gamma$ is a sequence converging to a point in $\del(\Gamma,\mathcal{P})$, where each $g_i$ is a power of a generator and $g_i$ and $g_j$ are not powers of a a common generator whenever $|i-j|=1$, then $\angle \left(U_p(g_1 \cdots g_{i-1}), S_{d-p}(g_i) \right) \geq \alpha_0$ for $p\in\{1,d-1\}$ and for all $n$.

Suppose this were not true, so that
there exist 
\begin{itemize}
    \item a generator $s$, 
    \item a divergent sequence $(k_n)$ of integers, and 
    \item a divergent sequence of words $(w_n)$ of words in $\Gamma$ not starting in $s^{\pm1}$, which without loss of generality --- passing to a subsequence if needed --- converge to some point in $\del(\Gamma,\mathcal{P})$,
\end{itemize}
such that 
\[ \angle(U_1(\rho(w_n)), S_{d-1}(\rho(s^{k_n})) ) \leq 2^{-n} ;\]
then, in the limit, we obtain \[ \angle \left( \lim_{n\to\infty} U_1(\rho(w_n)), \lim_{n\to\infty} S_{d-1}(\rho(s^{k_n})) \right) = 0\]
but this contradicts transversality, since, by our hypothesis that none of the words $w_n$ starts with $s$, we must have $\lim w_n \neq \lim s^{k_n}$ as $n\to\infty$.

Thus we do have a uniform lower bound $\alpha_0 \leq \alpha$ as desired, and then Lemma \ref{lem:BPSA7}, together with the existence of a proper polynomial $\ubar{q}$ such that
\[ \frac{\sigma_1}{\sigma_2}(u_j^n) \geq \ubar{q}(n) \]
for all $j$,
tells us that $\log\frac{\sigma_p}{\sigma_{p+1}} (\gamma)$ grows at least linearly in $|\gamma|_c$, which gives us (D$-$).

We then conclude, by Theorem \ref{thm:gaps+maps}, that $\iota\colon \Gamma \into \PGL(d,\real)$ is $P_1$-dominated relative to $\mathcal{P}$.
\end{eg}

\printbibliography

@article{BPS,
    AUTHOR = {Bochi, Jairo and Potrie, Rafael and Sambarino, Andr\'{e}s},
     TITLE = {Anosov representations and dominated splittings},
   JOURNAL = {J. Eur. Math. Soc. (JEMS)},
  FJOURNAL = {Journal of the European Mathematical Society (JEMS)},
    VOLUME = {21},
      YEAR = {2019},
    NUMBER = {11},
     PAGES = {3343--3414},
   MRCLASS = {22E40 (20F67 37B99 37D30 53C35)},
  MRNUMBER = {4012341},
       DOI = {10.4171/JEMS/905},
}

@article{GGKW,
shorthand = "GGKW",
author = "Guéritaud, François and Guichard, Olivier and Kassel, Fanny and Wienhard, Anna",
doi = "10.2140/gt.2017.21.485",
fjournal = "Geometry & Topology",
journal = "Geom. Topol.",
number = "1",
pages = "485--584",
publisher = "MSP",
title = "Anosov representations and proper actions",
volume = "21",
year = 2017
}

@article{KLP,
	title = {Anosov subgroups: dynamical and geometric characterizations},
	volume = {3},
	doi = {10.1007/s40879-017-0192-y},
	number = {4},
	journal = {Eur. J. Math.},
	author = {Kapovich, Michael and Leeb, Bernhard and Porti, Joan},
	year = {2017},
	mrnumber = {3736790},
	pages = {808--898}
}

@book{osinRH,
title={Relatively Hyperbolic Groups: Intrinsic Geometry, Algebraic Properties, and Algorithmic Problems},
author={Osin, Denis V.},
number={v. 179, no. 843},
isbn={9780821838211},
lccn={2005053663},
series={American Mathematical Society},
year={2006},  
publisher={American Mathematical Society}
}

@Article{GrovesManning,
author={{Groves}, Daniel and {Manning}, Jason F.},
title="Dehn filling in relatively hyperbolic groups",
journal="Israel J.  Math.",
year="2008",
volume="168",
number="1",
pages="317",
doi="10.1007/s11856-008-1070-6",
}

@article{Bowditch,
author = {{Bowditch}, Brian H.},
title = {Relatively hyperbolic groups},
journal = {International Journal of Algebra and Computation},
volume = {22},
number = {03},
pages = {1250016},
year = {2012},
doi = {10.1142/S0218196712500166},
}

@article{CLSS,
    AUTHOR = {Canary, Richard D. and Lee, Michelle and Stover, Matthew},
     TITLE = {Amalgam {A}nosov representations},
      NOTE = {With an appendix by Canary, Lee, Andr\'{e}s Sambarino and Stover},
   JOURNAL = {Geom. Topol.},
  FJOURNAL = {Geometry \& Topology},
    VOLUME = {21},
      YEAR = {2017},
    NUMBER = {1},
     PAGES = {215--251},
   MRCLASS = {37D20 (20F67 37B05 37F30 57M50)},
       DOI = {10.2140/gt.2017.21.215},
}

@article{Groff,
Author = {Groff, Bradley W.},
Title = {Quasi-isometries, boundaries and {JSJ}-decompositions of relatively hyperbolic groups},
FJournal = {{Journal of Topology and Analysis}},
Year = 2013,
Volume = 5,
Number = 4,
pages = {451-475},
Month = dec,
Publisher = {{World Scientific}},
DOI = {10.1142/S1793525313500192},
Journal = {{J. Topol. Anal.}},
}

@online{KL,
   author = {{Kapovich}, Michael and {Leeb}, Bernhard},
    title = "{Relativizing characterizations of Anosov subgroups, I}",
  eprinttype = {arXiv},
   eprint = {1807.00160},
 primaryClass = "math.GR",
    year = 2018,
    month = jun,
}

@online{KasselPotrie,
    title={Eigenvalue gaps for hyperbolic groups and semigroups},
    author={Fanny Kassel and Rafael Potrie},
    year=2020,
    month=feb,
    eprinttype = {arXiv},
    eprint={2002.07015},
    primaryClass={math.DS}
}

@online{reldomreps,
       author = {{Zhu}, Feng},
        title = {Relatively dominated representations},
         year = 2019,
        month = dec,
    eprinttype = {arXiv},
       eprint = {1912.13152},
 primaryClass = {math.GR},
       adsurl = {https://ui.adsabs.harvard.edu/abs/2019arXiv191213152Z}
}

@online{Kostas_AnosovWEG,
    title={Anosov representations, strongly convex cocompact groups and weak eigenvalue gaps},
    author={Konstantinos Tsouvalas},
    year=2020,
    month=aug,
    eprint={2008.04462},
    archivePrefix={arXiv},
    primaryClass={math.GT}
}

@article {GP_FloydRH,
    AUTHOR = {Gerasimov, Victor and Potyagailo, Leonid},
     TITLE = {Quasi-isometric maps and {F}loyd boundaries of relatively
              hyperbolic groups},
   JOURNAL = {J. Eur. Math. Soc. (JEMS)},
  FJOURNAL = {Journal of the European Mathematical Society (JEMS)},
    VOLUME = {15},
      YEAR = {2013},
    NUMBER = {6},
     PAGES = {2115--2137},
   MRCLASS = {20F67},
       DOI = {10.4171/JEMS/417},
}

@book {CDP,
    AUTHOR = {Coornaert, Michel and Delzant, Thomas and Papadopoulos, Athanase},
     TITLE = {G\'{e}om\'{e}trie et th\'{e}orie des groupes},
    SERIES = {Lecture Notes in Mathematics},
    VOLUME = {1441},
      NOTE = {Les groupes hyperboliques de Gromov. [Gromov hyperbolic
              groups],
              With an English summary},
 PUBLISHER = {Springer-Verlag, Berlin},
      YEAR = {1990},
     PAGES = {x+165},
      ISBN = {3-540-52977-2},
}

@article {GerFloyd,
    AUTHOR = {Gerasimov, Victor},
     TITLE = {Floyd maps for relatively hyperbolic groups},
   JOURNAL = {Geom. Funct. Anal.},
  FJOURNAL = {Geometric and Functional Analysis},
    VOLUME = {22},
      YEAR = {2012},
    NUMBER = {5},
     PAGES = {1361--1399},
DOI = {10.1007/s00039-012-0175-6},
}

@article {Ivan_thickFloyd,
    AUTHOR = {Levcovitz, Ivan},
     TITLE = {Thick groups have trivial {F}loyd boundary},
   JOURNAL = {Proc. Amer. Math. Soc.},
  FJOURNAL = {Proceedings of the American Mathematical Society},
    VOLUME = {148},
      YEAR = {2020},
    NUMBER = {2},
     PAGES = {513--521},
MRCLASS = {20F65 (57M07)},
  MRNUMBER = {4052190},
MRREVIEWER = {Sebastian Wolfgang Hensel},
       DOI = {10.1090/proc/14745},
}

@online{Max,
    title={A quantified local-to-global principle for Morse quasigeodesics}, 
    author={Max Riestenberg},
    year = 2021,
    month = jan,
    archivePrefix={arXiv},
    eprint = {2101.07162},
    primaryClass = {math.DG},
}

@article {Karlsson,
    AUTHOR = {Karlsson, Anders},
     TITLE = {Free subgroups of groups with nontrivial {F}loyd boundary},
   JOURNAL = {Comm. Algebra},
  FJOURNAL = {Communications in Algebra},
    VOLUME = {31},
      YEAR = {2003},
    NUMBER = {11},
     PAGES = {5361--5376},
       DOI = {10.1081/AGB-120023961},
}

@article {Floyd,
    AUTHOR = {Floyd, William J.},
     TITLE = {Group completions and limit sets of {K}leinian groups},
   JOURNAL = {Invent. Math.},
  FJOURNAL = {Inventiones Mathematicae},
    VOLUME = {57},
      YEAR = {1980},
    NUMBER = {3},
     PAGES = {205--218},
MRCLASS = {57M15 (22E40 30F40 51M20)},
  MRNUMBER = {568933},
MRREVIEWER = {I. Kra},
       DOI = {10.1007/BF01418926},
}

@incollection {Gromov,
    AUTHOR = {Gromov, M.},
     TITLE = {Hyperbolic groups},
 BOOKTITLE = {Essays in group theory},
    SERIES = {Math. Sci. Res. Inst. Publ.},
    VOLUME = {8},
     PAGES = {75--263},
 PUBLISHER = {Springer, New York},
      YEAR = {1987},
   MRCLASS = {20F32 (20F06 20F10 22E40 53C20 57R75 58F17)},
  MRNUMBER = {919829},
MRREVIEWER = {Christopher W. Stark},
       DOI = {10.1007/978-1-4613-9586-7_3},
}

@article {KLP_Morse,
    AUTHOR = {Kapovich, Michael and Leeb, Bernhard and Porti, Joan},
     TITLE = {A {M}orse lemma for quasigeodesics in symmetric spaces and
              {E}uclidean buildings},
   JOURNAL = {Geom. Topol.},
  FJOURNAL = {Geometry \& Topology},
    VOLUME = {22},
      YEAR = {2018},
    NUMBER = {7},
     PAGES = {3827--3923},
       DOI = {10.2140/gt.2018.22.3827},
}

@article {DGO,
    AUTHOR = {Dahmani, Fran\c{c}ois and Guirardel, Vincent and Osin, Denis},
     TITLE = {Hyperbolically embedded subgroups and rotating families in
              groups acting on hyperbolic spaces},
   JOURNAL = {Mem. Amer. Math. Soc.},
  FJOURNAL = {Memoirs of the American Mathematical Society},
    VOLUME = {245},
      YEAR = {2017},
    NUMBER = {1156},
     PAGES = {v+152},
      ISBN = {978-1-4704-2194-6},
       DOI = {10.1090/memo/1156},
}

@misc{ZZ,
      title={Relatively Anosov representations via flows I: Theory}, 
      author={Feng Zhu and Andrew Zimmer},
      year={2022},
      note={Preprint to appear}
}

\end{document}